\documentclass[11pt,twoside]{amsart}
\title{
Hodge-to-singular correspondence for reduced curves}
\author{Mirko Mauri and Luca Migliorini}
\usepackage[latin1]{inputenc}
\usepackage[T1]{fontenc}
\usepackage{fixltx2e}
\usepackage{amsthm}
\usepackage{mathrsfs}

\usepackage{mathtools}
\usepackage{graphicx}
\usepackage{longtable}
\usepackage{float}
\usepackage{wrapfig}
\usepackage{rotating}
\usepackage{amsmath}
\usepackage{textcomp}
\usepackage{marvosym}
\usepackage{wasysym}
\usepackage{amssymb}
\usepackage[hidelinks]{hyperref}
\usepackage{tikz}
\usepackage{tikz-cd}
\usepackage{float}
\usepackage{stmaryrd}
\usepackage{enumerate}
\usepackage{caption}
\usepackage{multirow}
\usepackage{fancyhdr}

\newcommand{\TBC}[1]{}

\usetikzlibrary{matrix}
\usetikzlibrary{shapes.multipart}
\usetikzlibrary{arrows}
\usepackage{amsthm,amsmath,amsfonts,amscd}
\usepackage{tikz-3dplot}
\usetikzlibrary{decorations.markings}
\usetikzlibrary{arrows}
\usepackage{parskip}
\allowdisplaybreaks
\newcommand{\Z}{\mathbb{Z}}

\newcommand{\Gm}{\mathbb{G}_m}

\newcommand{\CC}{\mathbb{C}}
\newcommand{\QQ}{\mathbb{Q}}
\newcommand{\RR}{\mathbb{R}}

\newcommand{\ZZ}{\mathbb{Z}}
\newcommand{\PP}{\mathbb{P}}

\newcommand{\Spec}{\operatorname{Spec}}
\newcommand{\Proj}{\operatorname{Proj}}

\newcommand{\Gr}{\operatorname{Gr}}

\newcommand{\Gl}{\operatorname{GL}}
\newcommand{\Sl}{\operatorname{SL}}

\newcommand{\defo}{\mathrm{Def}}
\newcommand{\Tot}{\mathrm{Tot}}

\newcommand{\ext}{\mathcal{E}xt}
\newcommand{\IC}{\operatorname{IC}}
\newcommand{\IH}{I\!H}

\newcommand{\unJB}{\overline{J}_{d,\mathcal{B}}}

\newcommand{\unJvers}{\overline{J}_{d,B}}
\newcommand{\unJverse}{\overline{J}_{e,B}}
\newcommand{\unJversp}{\overline{J}_{d',B}}
\newcommand{\unJBGamma}{\overline{J}^{\circ}_{d}(\Gamma)}
\newcommand{\unJBe}{\overline{J}_{e,\mathcal{B}}}
\newcommand{\Gn}{\Gamma_{\underline{n}}}

\newcommand{\ffrac}[2]{\ensuremath{\frac{\displaystyle #1}{\displaystyle #2}}}

\newcommand{\MDol}{{M}}

\newcommand{\Mukai}{M(S, v, H)}

\newcommand{\Hom}{\operatorname{Hom}}
\newcommand{\codim}{\operatorname{codim}}

\newcommand{\rank}{\operatorname{rank}}

\newcommand{\Aut}{\operatorname{Aut}}
\newcommand{\Pic}{\operatorname{Pic}}

\DeclareMathOperator*{\bigboxtimes}{\scalerel*{\boxtimes}{\textstyle\sum}}
\usepackage{scalerel}

\usepackage{verbatim}

\begingroup
\makeatletter
\@for\theoremstyle:=definition,remark,plain\do{%
\expandafter\g@addto@macro\csname th@\theoremstyle\endcsname{%
\addtolength\thm@preskip\parskip
}%
}
\endgroup
\usepackage{graphicx}
\usepackage[capitalise]{cleveref}
\newtheorem{thm}{Theorem}[section]
\newtheorem{lem}[thm]{Lemma}
\newtheorem{cor}[thm]{Corollary}
\newtheorem{defn}[thm]{Definition}
\newtheorem{prop}[thm]{Proposition}

\newtheorem{quest}[thm]{Question}

\theoremstyle{definition}
\newtheorem{notation}[thm]{Notation}
\newtheorem{exa}[thm]{Example}
\newtheorem{rmk}[thm]{Remark}
\crefname{thm}{Theorem}{Theorems}
\Crefname{thm}{Theorem}{Theorems}
\Crefname{thm}{Theorem}{Theorems}
\Crefname{thm}{Theorem}{Theorems}
\crefname{lem}{Lemma}{Lemmas}
\Crefname{lem}{Lemma}{Lemmas}
\crefname{Conjecture}{Conjecture}{Conjectures}
\Crefname{Conjecture}{Conjecture}{Conjectures}
\crefname{Corollary}{Corollary}{Corollaries}
\Crefname{Corollary}{Corollary}{Corollaries}
\crefname{Claim}{Claim}{Claims}
\Crefname{Claim}{Claim}{Claims}
\crefname{Proposition}{Proposition}{Propositions}
\Crefname{Proposition}{Proposition}{Propositions}
\crefname{Remark}{Remark}{Remarks}
\Crefname{Remark}{Remark}{Remarks}
\crefname{Definition}{Definition}{Definitions}
\Crefname{Definition}{Definition}{Definitions}
\crefname{Example}{Example}{Examples}
\Crefname{Example}{Example}{Examples}
\crefname{Exercise}{Exercise}{Exercises}
\Crefname{Exercise}{Exercise}{Exercises}

\newtheoremstyle{plain2}    
   {}            
   {}            
   {\itshape}    
   {}            
   {\bfseries}   
   {.}           
   {5pt plus 1pt minus 1pt}  
   {{\thmnumber{#1} \thmname{#2}{\thmnote{ (#3)}}}}          
   
   \usepackage{enumitem}

  \usepackage{fullpage}
  \usepackage{microtype}

\begin{document}
\maketitle
\begin{abstract}
   We study the summands of the decomposition theorem for the Hitchin system for $\Gl_n$, in arbitrary degree, over the locus of reduced spectral curves. A key ingredient is a new correspondence between these summands and the topology of hypertoric quiver varieties.
     In contrast to the case of meromorphic Higgs fields, the intersection cohomology groups of moduli spaces of regular Higgs bundles depend on the degree. We describe this dependence.  
\end{abstract}
\vspace{0.5 cm}

Let $C$ be a compact Riemann surface of genus $g$ with canonical bundle $\omega_C$. Let $d$ and $n$ be integers with $n>1$.

The \emph{Dolbeault moduli space} $\MDol(n,d)$ is the coarse moduli space which parametrises S-equivalence classes of semistable Higgs bundles on $C$ of rank $n$ and degree $d$, i.e.\ polystable pairs $(\mathcal{E}, \phi)$ consisting of a vector bundle $\mathcal{E}$ of rank $n$ and degree $d$ and a section $\phi \in H^0(C, \mathcal{E}nd(\mathcal{E})\otimes \omega_C)$, called \emph{Higgs field}; see \cite{Simpson1994}. The Dolbeault moduli space is equipped with a projective fibration called \emph{Hitchin fibration} 
\begin{align}\label{Hitchinfibration}
    \chi(n,d)\colon  \MDol(n,d) & \to A_{n}\coloneqq \bigoplus^{n}_{i=1} H^0(C, \omega_C^{\otimes i}), \\ (\mathcal{E}, \phi) & \mapsto \operatorname{char}(\phi) = (\mathrm{tr}(\phi), \mathrm{tr}(\Lambda^2 \phi), \ldots, \mathrm{tr}(\Lambda^i \phi), \ldots, \det(\phi)), \nonumber
\end{align}
which assigns to $(\mathcal{E}, \phi)$ the characteristic polynomial $\operatorname{char}(\phi)$ of the Higgs field $\phi$.

The cohomology of $\MDol(n,d)$ has been extensively studied in the literature, especially under the assumption that $n$ and $d$ are coprime, namely when $\MDol(n,d)$ is smooth; see for instance \cite{Hitchin1987, HauselRodriguez-Villegas2008, HauselThaddeus04, HauselThaddeus03, Heinloth2016, Mellit2019, GWZ20, deCataldoHauselMigliorini2012, Mellit2019, deCataldoMaulikShen2019}.  Recently studies about the intersection cohomology $\IH^*(\MDol(n,d), \QQ)$ of singular Dolbeault moduli spaces have started to emerge; see for instance \cite{Felisetti2018, FelisettiMauri2020, MaulikShen2020II, Mauri2020, Mauri2021, MaulikShen2021:IHSL, Davison2021, Davison2021a, KK2021, Mejia2022}.\footnote{The lists of works are not intended to be exhaustive at all. We limit to list a series of selected papers that inspired the current article.} From this viewpoint, the decomposition theorem for the Hitchin map is a key tool to investigate the intersection cohomology of $\MDol(n,d)$: it allows to decompose $\IH^*(\MDol(n,d), \QQ)$ into building blocks, which are cohomology of some perverse sheaves on $A_{n}$.

In this article, we establish a correspondence which we find surprising between the summands of the decomposition theorem for the Hitchin map of $\MDol(n,d)$, and the singularities of $\MDol(n,0)$; see \cref{thm:IHde}. We call it the \emph{Hodge-to-singular correspondence}. This correspondence offers a new perspective on the cohomology of Dolbeault moduli spaces, even in the smooth case. What is surprising is that this relation is not realised via a direct geometric correspondence between $\MDol(n,d)$ and $\MDol( n,0)$, like a deformation or an alteration, but, as we explain below, by comparing the Hodge theory of $\MDol(n,d)$ and the singularity theory of $\MDol(n,0)$ on the common base of the Hitchin fibration.

The Hodge-to-singular correspondence gives a formula to express the dependence of the intersection cohomology of $M(n,d)$ on the degree $d$. This should be compared with the degree independence of the cohomology of spaces of meromorphic Higgs bundle \cite[Thm 0.1, 0.4]{MaulikShen2020II}. It also extends the main theorem of \cite{deCataldoHeinlothMigliorini19} to arbitrary degree, and provide a conceptual and elegant explanation for the occurence of the cographic matroid in \cite[\S 6]{deCataldoHeinlothMigliorini19}. Further, it provides new evidence for Davison's conjectures in \cite[Conj. 5.6, Rem. 5.10]{Davison2021a}, now proved in \cite[Thm 1.42, Thm 1.43, Cor. 14.9]{davison2022bps}, which however appeared after the first version of this paper.

Our correspondence suggests a unifying approach to the study of smooth and singular Dolbeault moduli spaces that may be useful to tackle the topological mirror symmetry conjecture \cite{HauselThaddeus03, Mauri2021} and the P=W conjecture \cite{deCataldoHauselMigliorini2012, deCataldoMaulik2018} simultaneously in both settings. 

In a different context, similar ideas have already been successfully employed to compute the Hodge numbers of the exceptional O'Grady 10 example of compact hyperk\"{a}hler manifolds, out of the Hodge numbers of the Hilbert scheme of five points on a K3 surface; see \cite{deCataldoRapagnettaSacca19}.  Indeed, both spaces can be interpreted as special moduli spaces of sheaves on K3 surfaces corresponding to different degrees $d$ (or rather Euler characteristics), and the computation of their Betti numbers can be reduced to determining how their cohomology depends on $d$. In fact, the Hodge-to-singular correspondence, formulated here for Dolbeault moduli spaces, can be stated also for Mukai systems on moduli space of sheaves on K3 or abelian surfaces; see Section \ref{sec:Mukai}.

\section{Main results}\label{sec:mainresults}
\subsection{Hodge-to-singular correspondence}\label{sec:degreedependence} 
In order to state our main results, we need to recall some features of the singularities of $\MDol(n,d)$. A point of $\MDol(n,d)$ corresponds to a strictly polystable Higgs bundle
\[(\mathcal{E}_1, \phi_1)^{\oplus m_1}\oplus \ldots \oplus (\mathcal{E}_r, \phi_r)^{\oplus m_r},\]
where $(\mathcal{E}_i, \phi_i)$ are distinct stable Higgs bundles of slope $\deg \mathcal{E}_i/\rank \mathcal{E}_i = d/n$, and it is singular if $(r, m_r)\neq(1,1)$. In particular, the $r$-uples of positive integers $\underline{m}= \{m_i\}$ and $\underline{n}=\{n_i\}\coloneqq \{\rank \mathcal{E}_i\}$ satisfy the following relations
\begin{equation}\label{eq:constrain}
    n=\sum^r_{i=1} m_i n_i \qquad n_i \cdot \frac{d}{n}=\deg \mathcal{E}_i \in \ZZ. 
\end{equation}
For $\underline{m}, \underline{n}\in \ZZ^r_{>0}$ satisfying \eqref{eq:constrain}, define $M^{\circ}_{\underline{m},\underline{n}}(d)$ the locus of polystable Higgs bundles in $\MDol(n,d)$ of multirank $\underline{n}$ and multiplicity $\underline{m}$, and call $M_{\underline{m},\underline{n}}(d)$ its closure in $\MDol(n,d)$. The loci $M^{\circ}_{\underline{m},\underline{n}}(d)$ are the strata of a complex Whitney stratification of $\MDol(n,d)$ 
\begin{equation}\label{eq:Whitneystra}
    \MDol(n,d)=\bigsqcup_{\underline{m},\underline{n}} M^{\circ}_{\underline{m},\underline{n}}(d),
\end{equation}
and an analytic normal slice $W_{\underline{m},\underline{n}}$ through $M^{\circ}_{\underline{m},\underline{n}}(d)$ is isomorphic to a Nakajima quiver variety; see \cite[Proof of Prop. 2.10]{Mauri2020}.
The Whitney stratification \eqref{eq:Whitneystra} induces a filtration of the base $A_n$ of the Hitchin fibration 
by closed subsets $S_{\underline{m},\underline{n}} \coloneqq \chi(n,0)(M_{\underline{m},\underline{n}}(0))$, and a stratification
\begin{equation}\label{eq:typestratification}
    A_n=\bigsqcup S^{\circ}_{\underline{m},\underline{n}} \qquad \text{ with } S^{\circ}_{\underline{m},\underline{n}} \coloneqq S_{\underline{m},\underline{n}} \setminus \bigcup_{S_{\underline{m}',\underline{n}'} \subsetneq S_{\underline{m},\underline{n}}} S_{\underline{m}',\underline{n}'}.
\end{equation}
Alternatively, $S^{\circ}_{\underline{m},\underline{n}}$ is the locus of characteristic polynomials whose irreducible factors have degree $n_i$ and multiplicity $m_i$, i.e.\ $a\in S^{\circ}_{\underline{m},\underline{n}}$ decomposes as $a= \prod^r_{i=1} a^{m_i}_i$, where $a_i$ is an irreducible polynomial of degree $n_i$. The locus $S_{\underline{m},\underline{n}}$ is the closure of $S^{\circ}_{\underline{m},\underline{n}}$ in $A_{n}$.
Let $A^{\mathrm{red}}_n\coloneqq \bigsqcup_{\underline{m}=(1, \ldots, 1)} S^{\circ}_{\underline{m},\underline{n}} \subset A_{n}$ be the open subset of reduced characteristic polynomials. Define \[M(n,d)^{\mathrm{red}}\coloneqq \chi(n,d)^{-1}(A^{\mathrm{red}}_n)\text{ and } M_{\underline{m},\underline{n}}(d)^{\mathrm{red}}\coloneqq M_{\underline{m},\underline{n}}(d)\cap M(n,d)^{\mathrm{red}}.\]
When $\underline{m}=\{1, \ldots, 1\}$, we abbreviate the notation by omitting the subscript $\underline{m}$, e.g.\
$S_{\underline{n}}\coloneqq S_{\underline{m},\underline{n}}$.
Remarkably, $W_{\underline{n}}$ admits diffeomorphic symplectic resolutions $\widetilde{W}_{\underline{n}}$, 
whose top cohomology $H^{\dim W_{\underline{n}}}(\widetilde{W}_{\underline{n}}, \QQ)$ is a representation of the fundamental group of $S^{\circ}_{\underline{n}}$; see \cref{FactToricQuiver} and \cref{thm:localmodelJM}.\footnote{In the following, we identify $\widetilde{W}_{\underline{n}}$ with the hypertoric quiver variety $Y(\Gamma^{\pm}_{\underline{n}}, \theta)$. For convenience, we do not introduce hypertoric quiver varieties here, and we postpone their definition to Section \ref{sec:toric hyperkahler variety}.} Denote by $\mathcal{L}_{\underline{n}}$ the corresponding local system on $S^{\circ}_{\underline{n}}$ (which carries a pure Hodge structure of weight zero). 

The Hodge-to-singular correspondence is a splitting of the complex of mixed Hodge module $R\chi(n,e)_*\QQ_{M(n,e)}|_{A^{\mathrm{red}}_n}$ whose direct summands are controlled by the topology of the singularities $W_{\underline{n}}$. 
 
For an integer $d$ let $\mathcal P_d$ be the set of partitions  $\underline{n}=\{n_i\}$ of $n$ such that $n_i d/n \in \ZZ$ for every $i$. Note that this set is in bijection with the partitions of $q = \gcd(n,d)$, where the correspondence exchanges the partition $\{a_i\}$ of $q$ with the partition $\{a_i n/q\} \in \mathcal P_d$.

\begin{thm}[Hodge-to-singular correspondence]\label{thm:IHde}
Let $\langle \bullet \rangle=[-2 \bullet](-\bullet)$. Let $d$ and $e$ be integers, with $e$ coprime to $n$. There is an isomorphism in $D^bMHM_{\mathrm{alg}}(A^{\mathrm{red}}_n)$ or in $D^b(A^{\mathrm{red}}_n)$ (ignoring the Tate shifts):
\begin{equation}\label{eq:HtS}
    R\chi(n,e)_*\QQ_{M(n,e)}|_{A^{\mathrm{red}}_n}\simeq \bigoplus_{\underline{n} \in \mathcal{P}_d} 
    R\chi(n,d)_*\IC(M_{\underline{n}}(d),\chi(n,d)^*\mathcal{L}_{\underline{n}})\langle \codim S_{\underline{n}}\rangle|_{A^{\mathrm{red}}_n}.
\end{equation}
Taking cohomology, we obtain an isomorphism of mixed Hodge structures
\begin{equation}\label{eq:decompglobal}
   H^{k}(M(n,e)^{\mathrm{red}}, \QQ)\simeq \bigoplus_{\underline{n} \in \mathcal{P}_d} 
\IH^{k - 2\codim S_{\underline{n}}}(M_{\underline{n}}(d)^{\mathrm{red}},\chi(n,d)^*\mathcal{L}_{\underline{n}})(-\codim S_{\underline{n}}). 
\end{equation}
which respects the
perverse filtrations\footnote{ We recall the definition of perverse filtration in \eqref{eq:perversefiltration}.} carried by these vector spaces.
\end{thm}

Note that the summand of the RHS of \eqref{eq:HtS} corresponding to the trivial partition $\underline{n}=\{n\}$ is $R\chi(n,d)_*\IC(M(n,d),\QQ)$. This means that \eqref{eq:decompglobal} measures the difference between $H^{k}(M(n,e)^{\mathrm{red}}, \QQ)$ and $\IH^{k}(M(n,d)^{\mathrm{red}}, \QQ)$.

Singular Dolbeault moduli space for $d=0$ were introduced to study the topology of character varieties, to which they are homemorphic via the non-abelian Hodge correspondence. Historically, to avoid to deal with singular moduli spaces, the moduli problem has been slightly modified or twisted (cf. the "bait-and-switch" in \cite[\S 1]{HauselThaddeus04}), and the Dolbeault moduli spaces with $d$ coprime to $n$ have received more attention in the literature. The Hodge-to-singular correspondence suggests that the historical order, say from smooth to singular spaces, should be reversed at least conceptually: the intersection cohomology of singular Dolbeault moduli spaces should be considered as "primitive" objects into which the cohomology of smooth Dolbeault moduli spaces decomposes.

In \cite[Thm 0.1]{MaulikShen2020II}, Maulik and Shen proved an analogue of \cref{thm:IHde} for twisted Higgs bundles, i.e.\ pairs $(\mathcal{E}, \phi)$ where now $\phi$ is a global section of $\mathcal{E}nd(\mathcal{E})\otimes \mathcal{O}_C(D)$ for some effective divisor $D$ with $\deg D > 2g-2$, the so-called "Fano" condition. Unfortunately, the estimates in \cite[\S 4]{MaulikShen2020II} are not useful in our setting, namely $\mathcal{O}_C(D)\simeq \omega_C$ or "Calabi--Yau" condition; see \cite[\S 0.4]{MaulikShen2020II}, \cite[\S 11]{CL2016} and \cref{rmk:comparisonmeromhol}. This accounts for the existence of proper supports  for the complex $R\chi(n,e)_*\IC(M(n,e), \QQ)$ in $A_n$, which on the contrary do not appear under the "Fano" assumption by \cite[Thm. 0.4]{MaulikShen2020II}.

We can suggestively summarise the Hodge-to-singular correspondence in the following table.

\vspace{-0.5 cm}
\begin{table}[h]
\centering
\renewcommand{\arraystretch}{3}
   \begin{tabular}{ccc}
    \multicolumn{3}{c}{\textsc{Hodge theory vs singularities of $\MDol(n,d)$ }}\\
  $\operatorname{gcd}(n,d)=1$& \qquad $\operatorname{gcd}(n,d)\neq 1, d\neq 0$ \qquad & $d=0$\\\hline \vspace{-1 cm}\\ \shortstack{no singularities,\\more summands of dec. thm} & \shortstack{intermediate\\ behaviour\\$\qquad$\\$\qquad$} & \shortstack{more singular strata,\\ no proper summands of dec. thm}\\  
    \end{tabular}
    \end{table}

\subsection{Dependence of $\IH^*(\MDol(n,d), \QQ)$ on degree}\label{sec:dependence}
The Hodge-to-singular correspondence asserts that the intersection cohomology $\IH^{k}(M(n,d)^{\mathrm{red}}, \QQ)$ for arbitrary $d \in \ZZ$ decomposes into the direct sum of isotypic components of tensor products of $\IH^{k}(M(n',0)^{\mathrm{red}}, \QQ)$ for $n'\leq n$ under the action of a symmetric group; see \cref{rmk:repIH}. We can refine the result by showing that $\IH^{k}(M(n',0)^{\mathrm{red}}, \QQ)$ is determined only by the smooth locus of the Hitchin fibration, which in particular implies that it cannot be further decomposed according to the decomposition theorem for the Hitchin fibration. Remarkably, this is in contrast to the smooth case, namely $d$ coprime to $n$, where $R\chi(n,d)_*\QQ_{M(n,d)}$ admits summands properly supported on $A_{n}$ which cannot be detected on the smooth locus of Hitchin fibration.

To this end, recall that any $a \in A_n$ can be viewed as a monic polynomial of degree $n$ with coefficient of the degree $n-i$ term in $H^0(C, \omega^{\otimes i}_C)$, or equivalently as the curve $C_{a} \subset \operatorname{Tot}(\omega_C)$ of degree $n$ over $C$ cut by this monic equation, called \emph{spectral curve}. So an alternative definition of $S^{\circ}_{\underline{n}}$ is the locus parametrising spectral curves $C_{a}=\bigcup^r_{i=1} C_i$ such that the reduced irreducible component 
$C_i$ 
has degree $n_i$ over $C$. We define also the following spaces:
\begin{enumerate}[label=(\roman*)]
    \item $S^\times_{\underline{n}}\subseteq S^\circ_{\underline{n}}$ is the locus parametrising reducible nodal curves having smooth irreducible components of degree $n_i$ over $C$;
    \item $\pi_{\underline{n}} \colon \mathcal{C}^{\times}_{\underline{n}} \to S^\times_{\underline{n}}$ is the universal spectral curve over $S^\times_{\underline{n}}$;
    \item $\mathrm{Pic}^0(\mathcal{C}^{\times}_{\underline{n}}) \to S^\times_{\underline{n}}$ is the (relative) Jacobian of the universal spectral curve over $S^\times_{\underline{n}}$;
    \item $g_{\underline{n}}: A^{\times}_{\underline{n}} \to S^\times_{\underline{n}}$ is the maximal abelian proper quotient of $\mathrm{Pic}^0(\mathcal{C}^{\times}_{\underline{n}})$, or equivalently the Jacobian of the normalization of $\mathcal{C}^{\times}_{\underline{n}}$;
    \item $\Lambda^l_{\underline{n}} \coloneqq R^l(g_{\underline{n}})_* \QQ_{A^{\times}_{\underline{n}}}= \Lambda^l R^1(g_{\underline{n}})_* \QQ_{A^{\times}_{\underline{n}}}= \Lambda^l R^1(\pi_{\underline{n}})_* \QQ_{A^{\times}_{\underline{n}}}$; see \cite[Lemma 1.3.5]{deCataldoHauselMigliorini2012}.
\end{enumerate}

\begin{defn}
The \emph{Ng\^{o} strings}\footnote{The definition of Ng\^{o} strings differs from \cite[Def. A.0.4]{deCataldoRapagnettaSacca19} by a shift and a Tate twist. This convention avoids to write the shifts all the times.} for the partition $\underline{n}$ of $n$ is 
 the complex of mixed Hodge module supported on $S_{\underline{n}}$ given by 
\begin{equation}\label{eq:Ngo}
\mathscr{S}_{\underline{n}} \coloneqq \bigoplus^{2 \dim S_{\underline{n}}}_{l=0} \IC(S_{\underline{n}}, \Lambda^l_{\underline{n}}
\otimes \mathcal{L}_{\underline{n}})[-l]\langle \codim S_{\underline{n}} \rangle. 
\end{equation}
 \end{defn}
 Note that $\mathscr{S}_{\underline{n}}$ is independent of $d$, since $\mathcal{L}_{\underline{n}}$ and $\Lambda^l R^1(g_{\underline{n}})_* \QQ_{A^{\times}_{\underline{n}}}$ are so. 

 In the coprime case $\mathrm{gcd}(e,n)=1$, the summands of the decomposition theorem of the Hitchin fibration on the locus of reduced spectral curves have been described in \cite{deCataldoHeinlothMigliorini19}. As a first application of the Hodge-to-singular correspondence, we recover \cite[Prop. 4.1, Thm 6.12]{deCataldoHeinlothMigliorini19}. 
\begin{thm}[Ng\^{o} strings in degree $e$ coprime to $n$]\label{cor:degree1}
Let $e$ be an integer coprime to $n$. There is an isomorphism in $D^bMHM_{\mathrm{alg}}(A^{\mathrm{red}}_n)$ or in $D^b(A^{\mathrm{red}}_n)$ (ignoring the Tate shifts):
\begin{equation}
R\chi(n,e)_*\QQ_{\MDol(n,e)}|_{A^{\mathrm{red}}_n}\simeq \bigoplus_{\underline{n}\vdash \, n} \mathscr{S}_{\underline{n}}|_{A^{\mathrm{red}}_n}.    
\end{equation}
\end{thm}
Note that the proof of \cref{cor:degree1} below is not independent of \cite{deCataldoHeinlothMigliorini19}. What is new is the conceptual interpretation of that result: the rather mysterious occurrence of the cohomology of the cographic matroid in \cite[Thm 6.12]{deCataldoHeinlothMigliorini19} can be now interpreted in geometric terms as the local intersection cohomology of the singularities of $\MDol(n,0)$, or the top cohomology of their resolutions, in symbols  $H^{\dim W_{\underline{n}}}(\widetilde{W}_{\underline{n}}, \QQ)$. In the follow-up paper \cite{MMP2022}, we provide a combinatorial and alternative proof of Theorems \ref{cor:degree1} and \ref{thm:Ngostring0}, which is independent of \cite{deCataldoHeinlothMigliorini19}.

In view of \cite{deCataldoHeinlothMigliorini19}, where the locus of reduced spectral curves carries a plethora of proper summands of the decomposition theorem, it could appear bizarre that instead there is none if $d=0$.

\begin{thm}[Ng\^{o} strings in degree 0] \label{thm:Ngostring0} The complex $R\chi(n,0)_*\IC(\MDol(n,0),\QQ)$ has full support on the reduced locus $A^{\mathrm{red}}_n$, i.e.\
\[
R\chi(n,0)_*\IC(\MDol(n,0),\QQ)|_{A^{\mathrm{red}}_n} \simeq \mathscr{S}_{\{n\}}|_{A^{\mathrm{red}}_n} 
\]
\end{thm}

\begin{quest}[Full support]\label{quest:fullsupport}
Does $R\chi(n,0)_*\IC(\MDol(n,0),\QQ)$ have full support on the whole Hitchin base $A_{n}$, i.e.\
\begin{equation*}
R\chi(n,0)_*\IC(\MDol(n,0),\QQ) \simeq
\mathscr{S}_{\{n\}}? 
\end{equation*}
\end{quest}

The multiplicative group $\mathbb{G}_m$ acts on $\MDol(n,d)$ by scaling the Higgs fields. Instead, the additive group $H^0(C, \omega_C)$ translates the Higgs fields, sending $(\mathcal{E}, \phi)$ to $(\mathcal{E}, \phi+ \omega \cdot \mathrm{id}_E)$ for all $\omega \in H^0(C, \omega_C)$. The $(H^0(C, \omega_C) \rtimes \mathbb{G}_m)$-actions descend to $A_{n}$, making the Hitchin fibration $(H^0(C, \omega_C) \rtimes \mathbb{G}_m)$-equivariant. In particular, the supports of $R\chi(n,d)_*\IC(\MDol(n,d),\QQ)$ are $(H^0(C, \omega_C) \rtimes \mathbb{G}_m)$-invariant. Then its minimal potential support is $H^0(C, \omega_C) \hookrightarrow \bigoplus^{n}_{i=1} H^0(C, \omega_C^{\otimes i})=A_{n}$, and it lies outside $A^{\mathrm{red}}_{n}$.

As positive partial answer to \cref{quest:fullsupport}, we observe that no summand of $R\chi(n,d)_*\IC(\MDol(n,d),\QQ)$ is supported on $H^0(C, \omega_C)$.
This generalises a result by Heinloth \cite[Thm 1]{Heinloth2016} to arbitrary degree. The proof relies on a result by Kinjo and Koseki \cite{KK2021} and Davison \cite{Davison2021} asserting that $R\chi(n,d)_*\IC(M(n,d),\QQ)$ is a direct summand $R\chi(n,e)_*\QQ_{M(n,e)}$ for $\gcd(n,e)=1$, see \cref{prop:inclusionofIC}. 

\begin{prop}[Vanishing of intersection form]\label{cor:vanishingintersectionform} The intersection form on $\IH^*(M(n,d), \QQ)$ is trivial. Equivalently the forgetful map
$
\IH_c^*(M(n,d), \QQ) \to \IH^*(M(n,d), \QQ) 
$ is 0.
\end{prop}

In \cref{thm:Fullsupportranktwo} we answer positively \cref{quest:fullsupport} in rank two and illustrate its far-reaching consequences. Note that the intersection cohomology groups of $M(2,d)$ have been computed in \cite[Thm 7.6]{Hitchin1987} and \cite[Thm 1.8, 1.9]{Mauri2020}. 

\begin{thm}[Ng\^{o} strings in rank two]\label{thm:Fullsupportranktwo}
There exist isomorphisms 
\begin{align*}
R\chi(2,1)_*\QQ_{\MDol(2,1)} \simeq \mathscr{S}_{\{2\}} \oplus \mathscr{S}_{\{1,1\}}, \qquad R\chi(2,0)_*\IC({\MDol(2,0)},\QQ) \simeq \mathscr{S}_{\{2\}},
\end{align*}
in the bounded derived category $D^b\mathrm{MMH}(A_n)$ of Hodge modules on $A_n$.
As a result, there exists a surjective map
\[H^*(\MDol(2,1),\QQ) \twoheadrightarrow IH^*(\MDol(2,0), \QQ)\]
strictly filtered by the perverse filtrations associated to $\chi(2,d)$.
\end{thm}

The Hodge-to-singular correspondence provides also a recursive strategy to express $\IH^*(\MDol^{\mathrm{red}}(n,d), \QQ)$ in any degree in terms of $\IH^*(\MDol^{\mathrm{red}}(n',0), \QQ)$ for $n' \leq n$. We outline it in Section \ref{rmk:dependence}. A closed non-recursive formula appear in \cite[Thm 1.1, Thm 1.6, Thm 1.8]{MMP2022}. In loc.\ cit.\ the authors and Pagaria determine explicitly the local systems in the Ng\^{o} strings for $R\chi(n,d)_*\IC(\MDol(n,d),\QQ)|_{A^{\mathrm{red}}_n}$. This completes the description of the summands of the decomposition theorem for the Hitchin fibration in arbitrary rank and degree over the locus of reduced spectral curve. 

An extension of the Hodge-to-singular correspondence to the whole $A_n$ would provide a recursive formula to compute $\IH^*(M(n,d), \QQ)$ for arbitrary degree $d$. After this paper have become available, Davison, Hennecart and Mejia achieved this latter goal; see \cite[\S 14.1]{davison2022bps}. Despite the different representation theoretic flavour of \cite{davison2022bps}, the current paper and \cite{davison2022bps} share ultimately the same central idea: to reduce to a local statement regarding Nakajima quiver varieties. The Hodge-to-singular correspondence has the virtue to offer a more geometric characterization of the key object in \cite{davison2022bps}, namely the BPS sheaf, as well as a closed formula for the BPS sheaf on each moduli space $M(n,d)$ rather than a generating formula; see \cref{prop:perverselocal}, \cref{prop:perverse} and \cite{MMP2022}. On the other hand, so far the Hodge-to-singular correspondence has been established only on the locus of reduced characteristic polynomials, while the results of \cite{davison2022bps} hold unconditionally over the whole Hitchin base. Observe however that the work of Davison, Hennecart and Mejia does not prescribe the shape of the summands of the decomposition theorem of the Hitchin fibration (for instance it does not answer \cref{quest:fullsupport}), which would be of independent interest for applications to the P=W and topological mirror symmetry conjecture. 

 Moreover, the Hodge-to-singular correspondence implies that $\IH^*(\MDol^{\mathrm{red}}(n,d), \QQ)$ depends only on the $\gcd(n,d)$; see \cref{cor:independence}. The independence of the cohomology ring $H^*(M(n,e), \QQ)$ for $\gcd(n,e)=1$ is a classical corollary of the non-abelian Hodge correspondence, observed for instance in \cite[Remark 4.8]{Hausel13}; see also \cite[Theorem 1.7]{GWZ20}, \cite[Corollary 1.2]{Mellit2020} and \cite{Yu2022}. The independence of the complex $R \chi(n,e)_* \IC(\MDol(n,e), \QQ)|_{A^{\mathrm{red}}_{n}}$ for $\gcd(n,e)=1$ has been proved using vanishing cycle techniques in \cite[Theorem 0.5, Remark 4.9]{MaulikShen2020I}; see also \cite[Theorem 1.1, Example 5.18]{KK2021} and the comment \cite[\S 1.5.(2)]{KM2023}. The latest and most refined algebraic proof \cite[Theorem 0.1]{dCMSZ2021} asserts that a canonical isomorphism $H^*(M(n,e), \QQ) \simeq H^*(M(n,e'), \QQ)$ for $\gcd(n,e)=\gcd(n,e')=1$ preserves simultaneously the ring structure, tautological classes, and perverse filtrations.  

In the singular case, the same argument in \cite[Remark 4.8]{Hausel13} (cf \cite[\S 0.4]{FSY2022}) shows that  \begin{equation}\label{eq:independglobal}
    \IH^*(M(n,d), \QQ) \simeq \IH^*(M(n,d'), \QQ)
\end{equation} for any integers $d$ and $d'$ with $\gcd(n,d)=\gcd(n,d')$. Then \cref{cor:independence} promotes the independence \eqref{eq:independglobal} at the sheaf level on the locus of reduced characteristic polynomials. As a corollary, this makes the isomorphism $\IH^*(\MDol^{\mathrm{red}}(n,d), \QQ) \simeq \IH^*(\MDol^{\mathrm{red}}(n,d'), \QQ)$ filtered with respect to the perverse filtration. 
\begin{cor}\label{cor:independence}
Let $d$ and $d'$ be integers with $\mathrm{gcd}(d,n)=\mathrm{gcd}(d',n)$. Then we have
\[R\chi(n,d)_*\IC(M(n,d), \QQ)|_{A^{\mathrm{red}}_n}\simeq R\chi(n,d')_*\IC(M(n,d'), \QQ)|_{A^{\mathrm{red}}_n}.\]
\end{cor}
Note that \cref{cor:independence} has been recently extended on the whole Hitchin base in \cite[Cor. 14.9]{davison2022bps}. \newpage

\subsection{Some applications of the Hodge-to-singular correspondence}
\subsubsection{Stabilisation} In \cite{CK2018} Coskun and Woolf have proposed a conjecture about the stabilization of the cohomology of moduli space of sheaves on compact complex surface. As a final remark, we observe that $\IH^*(M(n,d), \QQ)$ does stabilise as the rank $n$ grows, independently of the degree $d$.

\begin{prop}[Stabilisation]\label{cor:stab} The intersection Betti and Hodge numbers of $M(n,d)$ stabilize to the
stable intersection Betti and Hodge numbers of $M(n,0)$ as $n$
tends to infinity, i.e.\ the limit
\begin{equation}\label{eq:limstab}
    \lim_{n \to \infty} \dim \IH^*(M(n,d))^{p,q}
\end{equation}
is finite and independent of $d$.
\end{prop}
Analogous stabilisation results for the Mukai systems on moduli space of sheaves on K3 surfaces are discussed in \cref{sec:Mukai}.
\subsubsection{Restriction to smooth fibres} Let $M$ be a projective irreducible holomorphic symplectic variety of dimension $2k$ equipped with a Lagrangian fibration $f\colon M \to B$. In \cite{FSY2022} Felisetti, Shen and Yin have proved that the restriction of $H^*(M, \QQ)$ to a smooth fiber of $f$ is isomorphic to $H^*(\PP^k, \QQ)$. In general, this is false for non-compact hyperk\"{a}hler varieties, but it holds true for the Hitchin fibration for $n>1$ by \cite[Thm 5]{Baraglia2018}, up to a copy of $H^{\bullet}(C, \QQ)$. Here we propose an alternative proof.
\begin{prop}\label{prop:restr}
The restriction of $\IH^*(M(n,d))$ to any smooth fiber $M_a \subset M(n,d)$ of the Hitchin fibration $\chi(n,d)$ is
given by
\[
\mathrm{Im}\{\IH^*(M(n,d), \QQ) \to H^*(M_a, \QQ)\} = H^{\bullet}(C, \QQ) \otimes  {\QQ[\alpha|_{M_a}]}/{(\alpha|_{M_a}^{\dim M(n,d)+1})} 
\]
where $\alpha$ is a $\chi(n,d)$-relative ample class on $M(n,d)$.

Let $\check{M}(n,d)$ be a fibre of the determinant map $M(n,d) \to M(1, nd)$, given by $(\mathcal{E}, \phi) \mapsto (\det \mathcal{E}, \operatorname{tr} \phi)$. The restriction of $\IH^*(\check{M}(n,d))$ to any smooth fiber $\check{M}_a \subset \check{M}(n,d)$ of the Hitchin fibration $\check{\chi}(n,d)$ is
given by
\[
\mathrm{Im}\{\IH^k(\check{M}(n,d), \QQ) \to H^k(\check{M}_a, \QQ)\} = \begin{cases}
\langle \alpha^{k/2}|_{\check{M}_a}\rangle & \text{ if }k\text{ is even},\\
0 & \text{ if }k\text{ is odd},\\
\end{cases}
\]
\end{prop}

\subsection{Strategy} In this paper we propose a geometric proof of the Hodge-to-singular correspondence. Here we outline the main steps. It is well-known that the Hitchin fibration is a family of compactified Jacobians of spectral curves; see \cite{Hitchin1987a, BNR89, Schaub1998}. In Section \ref{sec:univcompactifiedJac}, we enlarge the Hitchin fibration $\chi(n,d) \colon M(n,d) \to A_n$ to a family $\pi \colon \unJvers \to B$ of compactified Jacobians over a versal deformation of spectral curves:\footnote{Actually, to make the diagram work we should slice the Hitchin base, restrict to the locus of nodal spectral curves, and  take into consideration the automorphisms of the curves by working over $\overline{\mathcal{M}}_{g}$.   For expository reasons, we allow this abuse here, and postpone the precise details to Section \ref{sec:univcompactifiedJac}.}
\begin{equation*}
\begin{tikzpicture}[baseline= (a).base]
\node[scale=1] (a) at (0,0){
\begin{tikzcd}
\MDol(n,d)  \ar[r]\ar[]{d}[swap]{\chi(n,d)} & \unJvers \ar[d, "\pi"] \\
    A_{n}\ar[r] & B,
\end{tikzcd}
};
\end{tikzpicture}
\end{equation*}
Note that the relative dimension of $\pi$ is strictly smaller than $\dim B$. This is a crucial ingredient to show that $R\pi_{*}\IC(\unJvers, \QQ)$ is independent of $d$, while $R\chi(n,d)_{*}\IC(M(n,d), \QQ)$ is not; see \cref{prop:fullsupportvers} and \cref{rmk:comparisonmeromhol}. In particular, we have
\begin{equation}\label{eq:independintro}
    R\pi_{*}\IC(\unJverse, \QQ)|_{A_{n}} \simeq R\pi_{*}\IC(\unJvers, \QQ)|_{A_{n}}.
\end{equation}
If $\mathrm{gcd}(n,e)=1$, the inclusion $\MDol(n,e) \hookrightarrow \unJverse$ is a regular embedding of smooth varieties, so 
\begin{equation}\label{eq:IC}
\IC(\unJverse, \QQ)|_{M(n,e)} \simeq  \QQ_{\unJverse}|_{M(n,e)} \simeq  \QQ_{M(n,e)},
\end{equation} 
and the LHS of \eqref{eq:independintro} is isomorphic to $R\chi(n,e)_{*}\QQ_{M(n,e)}$. If instead $\mathrm{gcd}(n,d)\neq 1$, the failure of \eqref{eq:IC} is measured by the topology of the singularities of $M(n,d)$. Following \cite{CMKV2015}, we identify the map $M(n,d) \hookrightarrow \unJvers$ with the inclusion of the hypertoric quiver variety $Y(\Gamma_{\underline{n}}, 0)$ into the toric Lawrence variety $X(\Gamma^{\pm}_{\underline{n}}, 0)$
\[
\begin{tikzpicture}[baseline= (a).base]
\node[scale=0.97] (a) at (0,0){
\begin{tikzcd}
  Y(\Gamma_{\underline{n}}, 0)\arrow[hookrightarrow, "\iota_{M}"]{r}\ar[d,"\chi_{\Gamma_{\underline{n}}}"'] & X(\Gamma_{\underline{n}}^{\pm}, 0) \ar[d, "\pi_{\Gamma_{\underline{n}}}"] \\
     \CC^{b_1(\Gamma_{\underline{n}})} \arrow[hookrightarrow, "\iota_A"]{r}& \CC^{s};
\end{tikzcd}
};
\end{tikzpicture}
\]
see \cref{thm:localmodelJM} for the proof, and Section \ref{sec:toric hyperkahler variety} for details about the notation. Now a local version of the Hodge-to-singular correspondence for hypertoric quiver varieties (\cref{prop:perverselocal}) shapes the RHS of \eqref{eq:independintro}, and gives the main result \cref{thm:IHde}; see also \cref{prop:perverse}. The key geometric input is the following: $X(\Gamma_{\underline{n}}^{\pm}, 0)$ admits small resolutions of singularities that restrict to semismall resolutions of $Y(\Gamma_{\underline{n}}, 0)$; see \cref{sec:decthm}. 

\subsection{Outline} 
\begin{itemize}
    \item In \cref{sec:Ngo} we prove a version of Ng\^{o} support theorem for weakly abelian Lagrangian fibrations on singular symplectic varieties; see \cref{thm:Ngostingsymplectic}. 
    We specialise the support theorem to the Hitchin fibration in \cref{thm:Ngostring}.
    \item In \cref{sec:toric hyperkahler variety} we discuss toric Lawrence varieties and hypertoric quiver varieties. The main result is the local Hodge-to-singular correspondence, i.e.\ \cref{prop:perverselocal}.
    \item In \cref{sec:univcompactifiedJac} we explain how locally \'{e}tale the Dolbeault moduli space embeds into the universal compactified Jacobian.
    \item In \cref{sec:singularitiesUnivComp} we describe local models for this embedding; see \cref{thm:localmodelJM}. To this purpose we need some auxiliary results: an explicit description of the image of the Kodaira--Spencer map for nodal spectral curves provided in \cref{sec:KS}; the topological trivialization of a tubular neighbourhood of an abelian variety in the universal compactified Jacobian in \cref{sec:tubularneighbou}.
    \item In \cref{sec:fullsupport} we show the degree independence of the intersection cohomology of the universal compactified Jacobian; see \cref{prop:fullsupportvers}.
    \item In \cref{sec:summary} we collect the proof of the main theorems stated in  \cref{sec:mainresults}.
\end{itemize}

\subsection{Notation} 
The intersection cohomology of a complex variety $X$ with middle perversity and rational coefficients is denoted by $\IH^*(X, \QQ) \coloneqq H^*(\IC(X, \QQ))$, where $\IC(X, \QQ)$ is the perverse intersection cohomology complex of $X$ shifted by $-\dim X$. Ordinary singular cohomology with rational coefficients is denoted by $H^*(X, \QQ)$. Recall that they all carry mixed Hodge structures.   Intersection complexes and perverse sheaves on an algebraic stack are descents of perverse sheaves from a smooth atlas; see \cite{Joshua1991, LMB00, LO09}.  The formalism of six operations works in this context too. An adaptation of the theory of mixed Hodge modules to stacks is sketched in \cite[\S 2.2 and \S 2.3]{Davison2021}. However, since we only deal with Deligne--Mumford stacks several technical issues simplify drastically. One can proceed for instance along the lines of \cite[III.15]{KW2001}. We use it in \cref{sec:singularitiesUnivComp}.

\subsection{Acknowledgements}
We would like to thank Alastair Craw, Mark de Cataldo, Jochen Heinloth, Daniel Huybrechts, Davesh Maulik, Jesse Leo Kass, Tasuki Kinjo, Nicola Pagani, Roberto Pagaria, Alex Perry, Giulia Sacc\`{a}, Junliang Shen, Filippo Viviani for numerous
discussions, emails and helpful advice. 
We wish to thank the anonymous referees for their careful reading and useful suggestions.

The first author was supported by the Max Planck Institute for Mathematics, the University of Michigan, the Hausdorff Institute of Mathematics in Bonn, and the Institute of Science and Technology Austria. This project has received funding from the European Union's Horizon 2020 research and innovation programme
under the Marie Sk{\l}odowska-Curie grant agreement No 101034413. The second author was supported by PRIN Project 2017 "Moduli and Lie theory".


\section{Stratifications}
 Recall that a \emph{stratification} of a complex variety \(X\) is a finite
collection of locally closed smooth subvarieties $X_i \subseteq X$, called \emph{strata}, such that $X$ is the disjoint union of $X_i$ with $i=1, \ldots, r$, i.e.\ $X = \bigsqcup^r_{i=1} X_i$.
A stratification is called a \emph{Whitney stratification} if all pairs of strata satisfy the Whitney conditions A and B.  We omit the precise definition of these conditions (see for instance \cite{GoreskyMacPherson88}), since in the following we will only use the stronger property \cref{defn:analytictrivial}, that implies Whitney conditions A and B by \cref{lem:analyticallytrivialWhitney}.

\begin{defn}\label{defn:analytictrivial}
 A stratification $X = \bigsqcup_i X_i$ is analytically trivial in the normal direction to each stratum, if for any $x \in X_i$ there exists a normal slice $N_x$ through $X_i$ at $x$, and a neighbourhood of $x$ in $X$ which is locally analytically isomorphic to $N_x \times T_x X_i$ at $(x, 0)$.
\end{defn}

\begin{lem}\label{lem:analyticallytrivialWhitney}
A stratification analytically trivial in the normal direction to each stratum is a Whitney stratification.
\end{lem}
\begin{proof}
We follow the argument in \cite[\S 4.7]{Mayrand2018}. Whitney conditions A and B are local, so it is enough to check it for $N_x \times T_x X_i$ at $(x,0)$. Since $T_x X_i$ is a smooth factor, it is enough to check it for $N_x$ at $x$. Then this follows from \cite[Lem. 19.3]{Whitney1965}. 
\end{proof}

\section{Ng\^{o} theorem for Lagrangian fibrations on singular spaces}\label{sec:Ngo}

The goal of this section is to generalise the description of Ng\^{o} strings in \cite[Thm 7.0.3]{deCataldoRapagnettaSacca19} to the case of weakly abelian Lagrangian fibrations on singular symplectic varieties.

We will follow closely Maulik and Shen \cite{MaulikShen2020II}. Let $B$ be a quasi-projective complex variety. Let $g \colon P \to B$ be a smooth $B$-group scheme with
connected fibers, and let $f \colon X\to B$ be a proper 
morphism. Assume
that the group scheme $P$ acts on X via
\begin{equation}\label{action}
    \mathrm{act} \colon P \times_B X \to X.
\end{equation}

\begin{defn}[Weak abelian fibration]\label{defn:weakabelian}
The triple $(X, P, B)$ is a weak abelian fibration of relative dimension $c$, if
\begin{enumerate}
    \item every fiber of the map $g$ is pure of dimension $c$, and $X$ has pure dimension
    \[\dim X = c + \dim B,\]
    \item the action \eqref{action} of $P$ on $X$ has affine stabilizers, and
    \item the Tate module $T_{\overline{\QQ}_l}(P)$ associated with the group scheme P is polarizable.
\end{enumerate}
\end{defn}

For a closed point $b \in B$, we define $\delta(b)$ as the dimension of the maximal affine and connected subgroup of $P_b$. For  any closed subvariety $Z \subseteq B$, we denote by $\delta_Z$ the minimum value of the (upper semi-continuous) function $\delta$ on $Z$.

\begin{thm}\label{thm:Ngofreeness}
Let $(X, P, B)$ be a weak abelian fibration of relative dimension $c$. Assume that 
\begin{enumerate}
    \item \emph{($\delta$-regularity)}\label{item:delta-reg} $\codim Z_{\delta} \geq \delta$ for any $Z_{\delta} \coloneqq \{b \in B \,|\, \delta(b)=\delta\}$;
    \item \label{itemII} $\tau_{>2c}(Rf_* \IC(X, \QQ))=0$ for the standard truncation functor $\tau_{>*}(-)$.
\end{enumerate}
Then every support $Z$ of $Rf_* \IC(X, \QQ)$ satisfies \[\codim Z = \delta(Z).\] 
\end{thm}
\begin{proof}
\cite[Theorem 1.8]{MaulikShen2020II} says that if $f \colon X \to B$ is part of the datum of a weak abelian fibration with the property \eqref{itemII}, then any support $Z$ of $Rf_* \IC(X, \QQ)$ satisfies $\codim Z \leq \delta_Z$. Together with the $\delta$-regularity, we obtain that 
$\codim Z = \delta_Z$. 
\end{proof}

\begin{prop}[Relative Dimension Bound] \label{prop:relative dimension bound general} 
Let $f: X \to B$ be an equidimensional proper morphism of complex algebraic varieties of relative dimension $c$. Suppose that there exists a Whitney stratification $X=\bigsqcup X_i$ such that for any $b \in B$ we have
\begin{equation}\label{eq:boundhyp}
    \codim_X X_i \geq 2 \codim_{f^{-1}(b)} (X_i \cap f^{-1}(b)).
\end{equation}
Then we have
\begin{itemize}
    \item $\tau_{>2c}(Rf_* \IC(X, \QQ))=0$;
    \item $R^{2c}f_* \IC(X, \QQ)= R^{2c}f_* \QQ_{X}$.
\end{itemize}
\end{prop}

\begin{proof}
Via proper base change we have
\[\mathcal{H}^{*}(Rf_*\IC(X, \QQ))_b = H^*(f^{-1}(b), \IC(X, \QQ)|_{f^{-1}(b)}) \text{ for any }b \in B.\]
These cohomology groups are the limit of the Grothendieck spectral sequence 
\[E^{p,q}_2 = H^{p-q}(f^{-1}(b), \mathcal{H}^q(\IC(X, \QQ))|_{f^{-1}(b)}) \Rightarrow H^p(f^{-1}(b), \IC(X, \QQ)|_{f^{-1}(b)}).\]
Let $Z_q \subseteq X$ be the support of the constructible sheaf $\mathcal{H}^q(\IC(X, \QQ))$.  In particular, we have
\[E^{p,q}_2=0 \text{ for }p-q > 2\dim (Z_q \cap f^{-1}(b)).\]
 The irreducible components of $Z_q$ are the closure of some $X_i$ with $q< \codim_X X_i$, or $X$ itself, if $q=0$. Indeed, the strong support condition for intersection cohomology implies that $\mathcal{H}^q(\IC(X, \QQ))_x =0$ for any $x \in X_{i}$ and $q\geq \max\{ \codim_X X_{i}, 1\}$; see \cite[(12)]{deCataldoMigliorini09}. By \eqref{eq:boundhyp}, we conclude that $E^{p,q}_2=0$ for
\[p> q+ 2\dim (Z_q \cap f^{-1}(b)) \geq \max\{ \codim_X X_{i}, 1\} + 2\dim (Z_q \cap f^{-1}(b)),\]
i.e.\ for $p>2c$ and arbitrary $q$, or for $p=2c$ and $q \neq 0$.
This yields 
\[\mathcal{H}^{>2c}(Rf_* \IC(X, \QQ))=0, \qquad \mathcal{H}^{2c}(Rf_* \IC(X, \QQ))= \mathcal{H}^{2c}(Rf_* \QQ_X).\] 
\end{proof}

\begin{defn}[Symplectic variety]
    A normal variety $X$ is \emph{symplectic} if it has rational singularities, and it admits a holomorphic symplectic form on its smooth locus $X^{\mathrm{reg}} \subset X$, i.e.\ a non-degenerate holomorphic (closed) 2-form $\omega \in H^0(X^{\mathrm{reg}}, \Omega^{2}_{X^{\mathrm{reg}}})$.     
\end{defn}
By \cite[Corollary 1.8]{KS2021}, this is equivalent to require that a holomorphic
symplectic form $\omega$ on $X^{\mathrm{reg}}$ extends to a (possibly degenerate) holomorphic 2-form $\widetilde{\omega}$ on a resolution $\widetilde{X} \to X$. We say that $X$ admits a \emph{symplectic resolution} if $\omega_{\widetilde{X}}$ is non-degenerate.

\begin{defn}[Lagrangian fibration]\label{defn:Lagrangian}
Let $X$ be a symplectic variety.
An irreducible subvariety $Y \subset X$, not contained in the singular locus of $X$, is \emph{isotropic} if $\omega|_{Y^{\mathrm{reg}} \cap X^{\mathrm{reg}}}$ vanishes. Note that $\dim Y \leq \frac{1}{2} \dim X$, and $Y$ is \emph{Lagrangian} if the equality holds. 

A \emph{Lagrangian fibration} is a proper surjective morphism $f \colon X \to B$  with connected fibers onto a normal
variety $B$ whose general fiber is Lagrangian. 
\end{defn}

Recall that a Lagrangian fibration $f \colon X \to B$ is equidimensional, and any irreducible component of a fiber of $X$ is not contained in the singular locus of $X$ and is Lagrangian; see \cite[Thm 17]{Schwald2020} where the projectivity assumption on $X$ and $B$ can be dropped.

\begin{prop}\label{prop:relationLagrangian}
Let $X$ be a symplectic variety endowed with a Lagrangian fibration $f: X \to B$. Then $X$ admits a complex Whitney stratification \begin{equation}\label{eq:whitneystrat}
    X = \bigsqcup X_i
\end{equation}
by locally closed disjoint  algebraic symplectic submanifolds $X_i$ such that the restriction map $f|_{X_i}: X_i \to B_i\coloneqq f(X_i)$ is Lagrangian. In particular, for any $b \in B$ we have
\begin{equation}\label{eq:dimensioncountfibreLagrangian}
\codim_X X_i = 2 \codim_{f^{-1}(b)} (X_i \cap f^{-1}(b)).
\end{equation}
\end{prop}
\begin{proof}
 
By \cite[Thm 2.3, Prop. 3.1]{Kaledin06}, $X$ admits a stratification by locally closed subsets $X_i$, with $i=1, \ldots, r$, such that:
\begin{itemize}
    \item $X_i$ are smooth and symplectic;
  \item $X_{i}$ is an open subset of the singular locus of the closure of some $X_j$;
    \item the stratification is analytically trivial in the normal direction to each stratum, so it is a Whitney stratification by \cref{lem:analyticallytrivialWhitney}.
\end{itemize}
Applying \cite[Thm 3.1]{Matsushita} iteratively, the restrictions $f|_{X_i}$ are Lagrangian. 
 
Since both $f$ and $f|_{X_i}$ are Lagrangian, we obtain
    $\dim X -\dim X_i = 2 \dim f^{-1}(b)- 2\dim (X_i \cap f^{-1}(b)).$
\end{proof}

\begin{exa} The Dolbeault moduli space $X=M(n,d)$ is a symplectic variety,\footnote{The Dolbeault moduli space $M(n,d)$ have rational singularities. A direct Riemann--Roch computation, analogous to \cite[Lem. 2.2]{BellamySchedler2019}, shows that the singular locus of $M(n,d)$ has codimension at least $4$, except for $(g,n,d)=(2,2,0)$ and $g=1$. In the general case, this implies that $M(n,d)$ has terminal rational singularities by \cite{Flenner1988} and \cite[Thm 5.22]{KollarMori1998}. In the special cases $(g,n,d)=(2,2,0)$ and $g=1$, the existence of a symplectic resolution implies the rationality of the singularities; see for instance \cite[Thm 1]{TU19931}, \cite[\S 6]{Beauville84} and \cite[Prop. 6.1]{FelisettiMauri2020}.} and the Hitchin system is a Lagrangian fibration; see for instance  \cite{Hitchin1987, Hitchin1987a} and \cite[Main Thm]{Markman1993}.
The Whitney stratification \eqref{eq:whitneystrat} for $M(n,d)$ is the stratification $\eqref{eq:Whitneystra}$. In particular, \cref{prop:relative dimension bound general} gives
\[\tau_{>\dim M(n,d)}(R\chi(n,d)_* \IC(\MDol(n,d), \QQ))=0, \qquad R^{\dim M(n,d)}f_* \IC(\MDol(n,d), \QQ)= R^{\dim M(n,d)}f_* \QQ_{\MDol(n,d)}.\]
\end{exa}

Now let 
$(X, P, B)$ be a $\delta$-regular weak abelian fibration. 
We denote by $I$ the set of supports $Z_{\alpha} \subseteq B$ for $R f_*\IC(X, \QQ)$. There exist open dense subsets $Z^{\times}_{\alpha}$ such that the restriction $P_{\alpha}$ of the group scheme $P$ admits a Chevalley decomposition 
\[0 \to R_{\alpha} \to P_{\alpha} \to A_{\alpha} \to 0,\]
where $R_{\alpha} \to Z^{\times}_{\alpha}$ is an affine group scheme and $g_{\alpha} \colon A_{\alpha} \to Z^{\times}_{\alpha}$ is a proper abelian group scheme. Let $\Lambda^{l}_{\alpha} \coloneqq R^l g_{\alpha *} \QQ_{A_{\alpha}}$.
\begin{thm}[Ng\^{o} strings from symplectic varieties]\label{thm:Ngostingsymplectic} Let $(X, P, B)$ be a $\delta$-regular weak abelian Lagrangian fibration on a symplectic variety $X$ of dimension $2c$. Then there exists an isomorphism in $D^bMHM_{\text{alg}}(B)$ (resp. $D^b(B, \QQ)$) 
\begin{equation}\label{eq:Ngofreeness}
   R f_*\IC(X, \QQ) \simeq \bigoplus_{\alpha \in I} \bigoplus^{2 \dim Z_{\alpha}}_{l=0} \IC(Z_{\alpha}, \Lambda^{l}_{\alpha} \otimes \mathscr{L}_{\alpha})[-l]\langle \codim(Z_{\alpha}) \rangle, 
\end{equation}
  where, for every $\alpha$, $\mathscr{L}_{\alpha}$ is a local system on an open subset of $Z_{\alpha}$.  
Moreover, we have:
\begin{enumerate}[label=(\roman*)]
    \item\label{eq:constraint} $\codim Z_{\alpha} = \delta(Z_{\alpha})$;
    \item\label{eq:topdimension} $R^{2c} f_*\QQ_X$ admits the following direct summand
$
\bigoplus_{\alpha \in I} i_{\alpha *}\mathscr{L}_{\alpha},
$
where $i_{\alpha} \colon Z^{\times}_{\alpha} \to B$ is the natural immersion.
\end{enumerate}
\end{thm}
\begin{proof}
    The constraint \ref{eq:constraint} is a combination of \cref{thm:Ngofreeness}, Propositions \ref{prop:relative dimension bound general} and  \ref{prop:relationLagrangian}. In \cite[Prop. 1.5]{MaulikShen2020II} Maulik and Shen provided a generalization of Ng\^{o} freeness \cite[Prop. 7.4.10]{Ngo2010} to the singular context. This gives the special form \eqref{eq:Ngofreeness} to the decomposition theorem for $R f_*\IC(X, \QQ)$. Taking  cohomology of \eqref{eq:Ngofreeness}, together with the isomorphism $R^{2c}f_* \IC(X, \QQ)= R^{2c}f_* \QQ_{X}$ in \cref{prop:relative dimension bound general}, we obtain \ref{eq:topdimension}. 
\end{proof}

\subsection{Ng\^{o} strings for Dolbeault moduli spaces}
We specialize \cref{thm:Ngostingsymplectic} to the Hitchin fibration $\chi(n,d) \colon M(n,d) \to A_n$. We rely on the recent key result by Kinjo and Koseki \cite{KK2021}, building on the work of Davison \cite{Davison2021}.

\begin{prop}\label{prop:inclusionofIC}
Let $d$ and $e$ be integers, with $e$ coprime to $n$. The complex 
$R\chi(n,d)_*\IC(M(n,d), \QQ)$ is a direct summand of $R\chi(n,e)_*\QQ_{M(n,e)}$.
\end{prop}
\begin{proof}
\cite[Thm 6.6]{Davison2021} and \cite[Cor. 5.15, Ex. 5.18]{KK2021}.
\end{proof}

\cref{prop:inclusionofIC} allows to extend \cite[Thm 4.1]{deCataldoHeinlothMigliorini19} and \cite[Thm 1]{Heinloth2016} to Dolbeault moduli spaces of arbitrary degree.
\begin{cor}\label{cor:support}
If $Z$ is a support of $R\chi(n,d)_*\IC(M(n,d), \QQ)$ with $Z \cap A^{\mathrm{red}}_n \neq \emptyset$, then $Z= S_{\underline{n}}$ for some partition $\underline{n}$ of $n$. In particular, the generic point of any support $Z$ of $R\chi(n,d)_*\IC(M(n,d), \QQ)$ with $Z \cap A^{\mathrm{red}}_n \neq \emptyset$ lies in the
locus $A^{\times}_n$ of nodal spectral curves.
\end{cor}
\begin{proof}
\cite[Prop 4.1]{deCataldoHeinlothMigliorini19}  and \cref{prop:inclusionofIC}. Alternatively, the same proof of \cite[Prop 4.1]{deCataldoHeinlothMigliorini19} works verbatim as long as the number of irreducible components of $\chi(n,d)^{-1}(a)$ is constant for $a \in S^{\circ}_{\underline{n}}$, which is proved in \cite[Prop 7.6.(3)]{MMP2022}.
\end{proof}

\begin{cor}[\cref{cor:vanishingintersectionform}]\label{cor:van} The intersection form on $\IH^*(M(n,d), \QQ)$ is trivial. Equivalently the forgetful map
$
\IH_c^*(M(n,d), \QQ) \to \IH^*(M(n,d), \QQ) 
$ is 0.
\end{cor}
\begin{proof}
By \cite[Prop. 4]{Heinloth2015} and the translation action of $H^0(C, \omega_C)$ in Section \ref{sec:dependence}, the vanishing of the intersection form is equivalent to show that $H^0(C, \omega_C) \subset A_n$ is not a support of $R\chi(n,d)_*\IC(M(n,d), \QQ)$. In the coprime case, this is proved by Heinloth in \cite[Thm 1]{Heinloth2016}. Together with \cref{prop:inclusionofIC}, we obtain the result in the non-coprime case.
\end{proof}
\begin{rmk}
\cref{cor:van} holds for moduli space of $\mathrm{PGL}_n$- and $\mathrm{SL}_n$-Higgs bundles too.
\end{rmk}
\begin{thm}[Ng\^{o} strings from Dolbeault moduli spaces in arbitrary degree] \label{thm:Ngostring} There exists an isomorphism in $D^bMHM_{\text{alg}}(A^{\text{red}}_n)$ (resp. $D^b(A^{\text{red}}_n, \QQ)$) 
\begin{equation}\label{eq:decthmNgo} R \chi(n,d)_* \IC(\MDol(n,d), \QQ)|_{A^{\text{red}}_n} \simeq \bigoplus_{\underline{n}\in I_d} \mathscr{S}(\mathscr{L}_{\underline{n}}(d))|_{A^{\text{red}}_n} ,
\end{equation}
where 
$I_d$ is the set of partitions $\underline{n}$ of $n$, 
and the Ng\^{o} string $\mathscr{S}(\mathscr{L}_{\underline{n}}(d))$ is the complex of mixed Hodge modules supported on $S_{ \underline{n}}$ given by
\[
\mathscr{S}(\mathscr{L}_{\underline{n}}(d)) \coloneqq \bigoplus^{2 \dim S_{ \underline{n}}}_{l=0} \IC(S_{\underline{n}}, \Lambda^l_{\underline{n}} 
\otimes \mathscr{L}_{ \underline{n}}(d))[-l]\langle \codim S_{\underline{n}}\rangle, 
\]
for certain local systems $\mathscr{L}_{ \underline{n}}(d)$ supported on an open set of $S_{\underline{n}}$.
\end{thm}
\begin{proof}
It follows from \cref{thm:Ngostingsymplectic} and \cref{cor:support}.
\end{proof}

The rest of the paper is devoted to the correspondence between the local systems $\mathscr{L}_{ \underline{n}}(d)$ and the singularities of $M(n,d)^{\mathrm{red}}$.

\section{Hypertoric quiver varieties}\label{sec:toric hyperkahler variety}
The singularities of $M(n,d)^{\mathrm{red}}$ are modelled on hypertoric quiver varieties. In this section we introduce these varieties and establish a local version of the Hodge-to-singular correspondence (\cref{prop:perverselocal}).


\subsection{Affine toric varieties associated to a quiver}\label{sec:affinetoricvarieties}
Let $Q = (V, E)$ be a connected directed graph (a \emph{quiver}) with $r$ vertices $V = \{v_1, . . ., v_r\}$ and $s$ oriented edges $E$. Given $e \in E$, let $s(e), t(e) \in V$ be the source and target of $e$, so that $e: s(e) \to t(e)$. The underlying unoriented graph, denoted by $|Q|$, has first Betti number $b_1 \coloneqq s-r+1$.

We consider the group of all $\ZZ$-linear combinations of $V$ whose coefficients
sum to zero. We identify this group with $\ZZ^{r-1}$, by fixing the basis $\{v_1-v_2, \ldots, v_1-v_r\}$. We also identify $\ZZ^{s}$ with the group of $\ZZ$-linear combinations of $E$. The boundary map of the quiver $Q$ 
\[A: \ZZ^s \to \ZZ^{r-1}, \quad e \mapsto s(e)-t(e)\]
induces an embedding of algebraic tori
\[\mathbb{T}^{r-1} \coloneqq \Spec \CC[\ZZ^{r-1}] \hookrightarrow \mathbb{T}^{s} \coloneqq \Spec \CC[\ZZ^{s}].\]
The torus $\mathbb{T}^{s}$ acts on the affine complex space
$\CC^{s}= \Spec \CC[z_e \colon e \in E]$,
and so does $\mathbb{T}^{r-1}$ by restriction. Explicitly, $\lambda = (\lambda_i)^r_{i=1} \in \mathbb{T}^{r}/\mathbb{T} = \mathbb{T}^{r-1}$ acts as $
    \lambda \cdot z_e = \lambda_{s(e)}z_e \lambda^{-1}_{t(e)}
$.
\begin{defn}
The affine toric variety $X(Q, 0)$ is the affine GIT quotient of $\CC^{s}$ by $\mathbb{T}^{r-1}$
\[X(Q, 0) \coloneqq \CC^{s}\sslash_0 \mathbb{T}^{r-1} =\Spec \CC[z_e : e \in E]^{\mathbb{T}^{r-1}}. \]
\end{defn}
The cone of the toric variety $X(Q, 0)$  is generated by the rows $\mathscr{B}=\{\beta_1, \ldots, \beta_{s}\} \subset \ZZ^{b_1}$ of a \emph{Gale dual} of $A$, i.e.\ a $(s \times b_1)$-matrix $B$ whose columns form a base of $\ker(A \colon \ZZ^{s} \to \ZZ^{r-1})\simeq \ZZ^{b_1}$, or equivalently a matrix that sits in the following exact sequence:
\[0 \to \ZZ^{b_1} \xrightarrow{B} \ZZ^s \xrightarrow{A} \ZZ^{r-1} \to 0.\]
Recall that a loop in a quiver is an edge whose head and tail coincide.
\begin{lem}\label{lem:loops}
Let $Q$ be a quiver with $m$ loops, and $Q'$ be the quiver obtained from $Q$ by deleting all the loops. Then we have
\[X(Q,0) \simeq X(Q',0) \times \CC^{m}.\]
\end{lem}
\begin{proof}
Let $L \subset E$ be the loops of $Q$. We identify the edges $E'$ of $Q'$ with $E \setminus L$. For any $l \in L$,  the coordinates $z_{l}$ are $\mathbb{T}^{r-1}$-invariant by construction, so we have
\[\CC[z_e \colon e \in E]^{\mathbb{T}^{r-1}} \simeq \CC[z_e \colon e \in E']^{\mathbb{T}^{r-1}} \otimes  \CC[z_{l} \colon l \in L].\]
\end{proof}

\subsection{Semiprojective toric varieties associated to a quiver} 
In general, $X(Q, 0)$ is a singular variety. A toric resolution of singularities can be obtained via projective GIT as follows.

The ring of polynomials $S\coloneqq \CC[z_e \colon e \in E]$ is graded by $\ZZ^{r-1}$ in such a way that
$\deg(z_e)=Ae$. For $\theta \in \ZZ^{r-1}$, let $S_{\theta}$ be the $\CC$-vector space of homogenous polynomials of degree $\theta$. For instance, $S_{0}$ is the coordinate ring of $X(Q, 0)$.

\begin{defn}
The semiprojective toric variety $X(Q, \theta)$ is the projective GIT quotient of $\CC^{s}$ by the torus $\mathbb{T}^{r-1}$ with respect to the stability $\theta \in \ZZ^{r-1}$, i.e.
\[X(Q, \theta) \coloneqq \CC^{s}\sslash_{\theta} \mathbb{T}^{r-1} \coloneqq \Proj \big( \bigoplus^{\infty}_{k=0} t^k S_{k \theta} \big). \]
\end{defn}

Consider now the coarsest complete fan $\Gamma(\mathscr{A})$ in $\RR^{r-1}\simeq \ZZ^{r-1}\otimes \RR$ that refines all complete simplicial fans in $\RR^{r-1}$ whose rays lie in $\mathscr{A}\coloneqq \{Ae \colon e \in E\}$. 
We say that $\theta \in \ZZ^{r-1}$ is generic if it lies in the interior of the maximal cones of $\Gamma(\mathscr{A})$.

\begin{prop}\emph{\cite[Prop. 8.2]{HauselStrumfels2002}}
For generic $\theta \in \Z^{r-1}$, the affinization morphism 
\[\pi_{X} = \pi_X (Q, \theta) \colon X(Q, \theta) \to \Spec H^0(X(Q, \theta), \mathcal{O}_{X(Q, \theta)})= X(Q, 0)\]
is a resolution of singularities.
\end{prop}

\subsection{Lawrence varieties and hypertoric quiver varieties}
We specialize the previous constructions to doubled quivers. Let $Q$ be a quiver with $r$ vertices, $s$ oriented edges $E$, and first Betti number $b_1$. 
\begin{defn}
The \emph{double} of $Q$, denoted by $Q^{\pm}$, is the quiver obtained from $Q$ by replacing each edge $e$ of $Q$ with a pair of edges $e^+$ and $e^-$ having the same endpoints as $e$ but opposite orientations.
\end{defn}
\begin{defn} 
The \emph{affine Lawrence toric variety} associated to $Q$ is the affine GIT quotient 
\[X(Q^{\pm}, 0) = \CC^{2s}\sslash_0 \mathbb{T}^{r-1}. \]
The \emph{Lawrence toric variety} associated to $Q$ is the projective GIT quotient 
\[X(Q^{\pm}, \theta) = \CC^{2s}\sslash_{\theta} \mathbb{T}^{r-1} = \Proj \big( \bigoplus^{\infty}_{k=0} t^k S_{k \theta} \big). \]
The action of $\mathbb{T}^{r-1}$ on $\CC^{2s}$ is prescribed by the boundary map of the double quiver $Q^{\pm}$ as in \S \ref{sec:affinetoricvarieties}.
\end{defn}
\begin{notation}
For convenience, we relabel the coordinates $z_{e^+}$ and $z_{e-}$ of $\CC^{2s}$ by $z_{e}$ and $w_{e}$, indexed by $e \in E$. Note that the definition of $X(Q^{\pm}, \theta)$ is independent of the choice of an orientation of $Q$: a different choice corresponds to exchange some coordinates $z_e$ with the corresponding $w_e$. 
\end{notation}

Suppose that the matrix $A=(a_{ie})$ corresponds to the boundary map for $Q$. We consider the homogeneous ideal
\begin{equation}\label{eq:circ}
    \mathrm{Circ}(\mathscr{B})\coloneqq \langle \sum_{e \in E} a_{ie}z_{e}w_{e} \colon i=2, \ldots r\rangle \subset S.
\end{equation}
\begin{defn}\label{def:torichyperkahlervarieties}
The \emph{hypertoric quiver variety} $Y(Q,0)$ (resp. $Y(Q,\theta)$) is the irreducible subvariety of the affine Lawrence variety $X(Q^{\pm},0)$ (resp. of the Lawrence variety $X(Q^{\pm},\theta)$) cut by the homogeneous ideal $\mathrm{Circ}(\mathscr{B})$.
\end{defn}

\begin{rmk}  
Hypertoric quiver varieties are both hypertoric varieties in the sense of \cite{BD2000, HP2004}, and \emph{Nakajima quiver varieties} whose dimension vector has all coordinates equals to one; see \cite{Nakajima98}.
In \cite{HauselStrumfels2002}, $Y(Q, 0)$ is called \emph{toric quiver variety}, but note that it is only rarely toric; see \cite[\S 10]{HauselStrumfels2002}. Moreover, Altmann and Hille called toric quiver varieties the fibre over $0$ of the morphism $\pi_Y\colon Y(Q, \theta) \to Y(Q, 0)$. To avoid confusion, we will refer to $Y(Q, 0)$ and $Y(Q, \theta)$ as hypertoric quiver variety.  
\end{rmk}

Since the product $z_e w_e$ is $\mathbb{T}^{r-1}$-invariant, the injective morphism of $\CC$-algebras
\[\CC[z_ew_e \colon e \in E] \hookrightarrow \CC[z_e, w_e \colon e \in E]^{\mathbb{T}^{r-1}}\] induces the fibre product square
\begin{equation}\label{eq:diagramsingular}
\begin{tikzpicture}[baseline= (a).base]
\node[scale=0.97] (a) at (0,0){
\begin{tikzcd}
  Y(Q, 0) =\Spec\bigg(\ffrac{\CC[z_e, w_e \colon e \in E]^{\mathbb{T}^{r-1}}}{\mathrm{Circ}(\mathscr{B})}\bigg)   \arrow[hookrightarrow, "\iota_{M}"]{r}\ar[d,"\chi_Q"'] & X(Q^{\pm}, 0)=\Spec\big(\CC[z_e, w_e \colon e \in E]^{\mathbb{T}^{r-1}}\big) \ar[d, "\pi_Q"] \\
     \CC^{b_1} = \Spec\bigg(\ffrac{\CC[z_e w_e \colon e \in E]}{\mathrm{Circ}(\mathscr{B})}\bigg) \arrow[hookrightarrow, "\iota_A"]{r}& \CC^{s} = \Spec\big(\CC[z_e w_e \colon e \in E]\big).
\end{tikzcd}
};
\end{tikzpicture}
\end{equation}
Hence, $Y(Q, 0)$ is the inverse image under the map $\pi_Q$ of the linear subspace $\CC^{b_1}$. 

\subsection{Cohomology of hypertoric quiver varieties} The cohomology of hypertoric quiver varieties can be expressed in terms of the cohomology of the cographic matroid. This permits to compare the Hodge-to-singular correspondence with \cite[Thm 6.11]{deCataldoHeinlothMigliorini19}.

\begin{defn} The \emph{cographic matroid} $\mathscr{C}_{Q}$ of the unoriented graph $|Q|$ is the matroid whose independent subsets are the subsets $I$ of unoriented edges of $|Q|$ such that $|Q| \setminus I$ is connected.
 The set $\mathscr{C}_{Q}$ is partially ordered with respect to inclusions and we denote by $|\mathscr{C}_{Q}|$ the associated simplicial complex, i.e.\ the complex whose $k$-dimensional faces are the independent subsets of cardinality $(k+1)$.
\end{defn}

We now recall some facts about the affinization morphisms
\[\pi_X: X(Q^{\pm}, \theta) \to X(Q^{\pm}, 0), \qquad \pi_Y: Y(Q, \theta) \to Y(Q, 0).\]
In particular, we observe that the top cohomology of $Y(Q, \theta)$ can be identified with the top cohomology of $|\mathscr{C}_{Q}|$,   which explains the occurence of the cographic matroid in \cite[\S 6]{deCataldoHeinlothMigliorini19}  via the Hodge-to-singular correspondence; see Sections \ref{sec:singularitiesUnivComp} and \ref{sec:summary}.  

In the following all the stability  $\theta, \theta' \in \ZZ^{r-1}$ are generic. 
\begin{prop}\label{FactToricQuiver}
The following facts hold:
\begin{enumerate}
    \item \emph{\cite[Prop. 8.2]{HauselStrumfels2002}} $X(Q^{\pm}, \theta)$ is a smooth toric variety of dimension $b_1+s$. 
    \item \emph{\cite[Prop. 8.2]{HauselStrumfels2002}} $Y(Q, \theta)$ is a smooth symplectic variety of dimension $2b_1$. 
    \item \emph{\cite[Lem. 6.4]{HauselStrumfels2002}} $\pi^{-1}_X(0)=\pi^{-1}_Y(0) \eqqcolon C(Q^{\pm}, \theta)$.
    \item \label{eq:Lagrangian} The irreducible components of $C(Q^{\pm}, \theta)$ are Lagrangian in $Y(Q, \theta)$. In particular, they all have dimension $b_1$.
    \item \label{item:CohoToricQuiver}\emph{\cite[Rmk after Lem. 6.5]{HauselStrumfels2002}} 
    The spaces $C(Q^{\pm}, \theta) \subset Y(Q, \theta) \subset X(Q^{\pm}, \theta)$ are deformation retracts in one another. 
    \item \emph{\cite[Lem. 2.1]{HP2004} \cite[Thm 6.3]{HauselStrumfels2002}} \label{rmk:diffeo} 
   The symplectic resolutions $Y(Q, \theta)$ and $Y(Q, \theta')$ are diffeomorphic.
    \item By \eqref{item:CohoToricQuiver} and \eqref{rmk:diffeo}, the cohomology 
    \[H^*(C(Q^{\pm}, \theta), \QQ)\simeq H^*(Y(Q, \theta), \QQ) \simeq H^*(X(Q^{\pm}, \theta), \QQ)\]
    is independent of $\theta$.
    \item \label{item:8}  \emph{\cite[Thm 7.3.3, 7.8.1]{Bjorner1992}}
There exists an isomorphism
\[H^{2b_1}(Y(Q, \theta), \QQ)\simeq \widetilde{H}^{2b_1-1}(|\mathscr{C}_Q|, \QQ).\]
\end{enumerate}
\end{prop}
\begin{proof}
We just comment on \eqref{eq:Lagrangian}. Let $F$ be an irreducible component of $C(Q^{\pm}, \theta)$. Since $C(Q^{\pm}, \theta)$ is a fibre of $Y(Q, \theta) \to Y(Q,0) \xrightarrow{\chi_{Q}} \CC^{b_1}$, then $\dim F \geq \dim Y(Q, \theta)-\dim \CC^{b_1}=b_1$ by semicontinuity.  
On the other hand, the symplectic form vanishes along $F^{\mathrm{reg}}$. Indeed, $F$ is a proper toric variety by (1) and (3). As all toric varieties, it have rational singularities, and its resolution $\widetilde{F}$ has no global differential forms, so $H^0(F^{\mathrm{reg}}, \Omega^2_{F^{\mathrm{reg}}})=H^0(\widetilde{F}, \Omega^2_{\widetilde{F}})=0$, see for instance by \cite[Cor. 1.8]{KS2021}. So $F$ is isotropic as in \cref{defn:Lagrangian}, and $\dim F \leq b_1$.
\end{proof}



\subsection{Decomposition theorem for hypertoric quiver varieties}\label{sec:decthm}
The goal of this section is to describe the restriction $\IC(X(Q^{\pm},0), \QQ)|_{Y(Q,0)}$; see \cref{prop:perverselocal}. To this end, we briefly recall the statement of the decomposition theorem for semismall maps.

Let $f \colon X \to Y$ be a proper morphism of irreducible varieties. A stratification of $f$ is a collection of finitely many locally closed subsets $Y_k$ such that $f^{-1}(Y_k) \to Y_k$ are topologically locally trivial fibrations.  A stratum $Y_k$ is \emph{relevant} if $2 \dim f^{-1}(Y_k) - \dim(Y_k)= \dim X$.
\begin{defn}[Semismall maps]
The map $f$ is said to be \emph{semismall} if 
\[2 \dim f^{-1}(Y_k) - \dim(Y_k)\leq \dim X \qquad \text{ for any } k\geq 0.\]
Further if the stronger inequalities \[2 \dim f^{-1}(Y_k) - \dim(Y_k)< \dim X\qquad \text{ for any } k> 0\]
hold, we say that $f$ is \emph{small}.
\end{defn}

\begin{thm}[Decomposition theorem for semismall maps]\label{DecThm}
Let $f\colon X \to Y$ be a semismall map from a smooth variety $X$. Then there exists a canonical isomorphism in $D^bMHM_{\text{alg}}(Y)$ (resp. $D^b(Y, \QQ)$) 
\[\mathrm{R}f_* \QQ_X \simeq \bigoplus_{Y_k} \IC\big(\overline{Y}_k, \mathrm{R}^{\codim Y_k}f_*\QQ_{f^{-1}(Y_k)}\big)\langle \codim Y_k/2 \rangle,\]
where the summation index runs over all the relevant strata of a stratification of $f$. 

In particular, if $f$ is small, then $\mathrm{R}f_* \QQ_X \simeq \IC(Y, \QQ)$.
\end{thm}

We note that, since, for every $k$, we have that  $f^{-1}(Y_k) \to Y_k$ is a  topologically locally trivial fibration, the sheaves   $\mathrm{R}^{\codim Y_k}f_*\QQ_{f^{-1}(Y_k)}$ are in fact local systems. The stalk at a point $y_k \in Y_k$ is the vector space generated by the irreducible components of $f^{-1}(y_k)$, therefore these local systems have finite monodromy.

Given a quiver $Q=(V,E)$, we  describe complex Whitney stratifications of the toric Lawrence variety $X(Q^{\pm},0)$ and of hypertoric quiver variety $Y(Q,0)$.

\begin{defn}
 A partition $\underline{V}$ of the set $V$ is a set $\{V_1, \ldots, V_r\}$ of disjoint subsets $V_i \subset V$ such that $V = \bigsqcup^r_{i=1} V_i$. The quiver $Q_{\underline{V}}$ is obtained from $Q$ by identifying all vertices in $V_i$, for any $i=1, \ldots, r$, and by deleting all loops.
\end{defn}

\begin{prop}\label{prop:toricLawrenceWhitney}
The toric Lawrence variety $X(Q^{\pm},0)$ admits a complex Whitney stratification
\begin{equation}\label{eq:whitneyI}
    X(Q^{\pm},0) = \bigsqcup_{\underline{V} \vdash V} X_{\underline{V}},
\end{equation}
by toric-invariant strata $X_{\underline{V}}$ whose analytic normal slice is isomorphic to $X(Q^{\pm}_{\underline{V}}, 0)$. This induces a complex Whitney stratification of $Y(Q,0)$ 
\begin{equation}\label{eq:whitneyII}
    Y(Q,0) = \bigsqcup_{\underline{V} \vdash V} Y_{\underline{V}} \qquad \text{with }Y_{\underline{V}} \coloneqq X_{\underline{V}} \cap Y(Q,0),
\end{equation}
whose normal slice is isomorphic to $Y(Q_{\underline{V}}, 0)$. 
\end{prop}
\begin{proof}
 Let $N$ be the lattice of one-parameter subgroups of the torus $\mathbb{T}^{b_1+s} = N \otimes \CC^*$.
 The affine toric variety $X(Q^{\pm}, 0)$ corresponds to the cone $\sigma \subset N \otimes \RR$ generated by the columns of a matrix $B^{\pm}$ Gale dual to the boundary map $A^{\pm}\coloneqq (A, -A)\colon \ZZ^{2s} \to \ZZ^{r-1}$ of the quiver $Q^{\pm}$; see \cref{sec:affinetoricvarieties}. 
 Let $\tau$ be a face of $\sigma$, $N_{\tau}$ be the sublattice of $N$ spanned by $\tau \cap N$, and $N(\tau)\coloneqq N/N_{\tau}$. The sets of rays of $\sigma$, $\tau$, and those of $\sigma$ not in $\tau$ are denoted respectively $\sigma(1)$, $\tau(1)$ and $\tau(1)^{c}$. The exact sequences
 \begin{equation*}
\begin{tikzpicture}[baseline= (a).base]
\node[scale=1] (a) at (0,0){
\begin{tikzcd}
0 \ar[r] & \ZZ^{\tau(1)} \ar[r] \ar[d, two heads, "(B^{\pm}_{\tau})^T"] & \ZZ^{\sigma(1)} \ar[r] \ar[d, two heads, "(B^{\pm})^T"]& \ZZ^{\tau(1)^c} \ar[d, two heads] \ar[r] & 0\\
0 \ar[r] & N_{\tau} \ar[r]  & N \ar[r] & N(\tau) \ar[r] & 0
\end{tikzcd}
};
\end{tikzpicture}
\end{equation*}
 are dual to
  \begin{equation*}
\begin{tikzpicture}[baseline= (a).base]
\node[scale=1] (a) at (0,0){
\begin{tikzcd}
0 \ar[r] & M(\tau) \ar[r] \ar[d, hookrightarrow] & \ZZ^{\tau(1)^{c}} \ar[r, "A^{\pm}|_{\tau(1)^c}"] \ar[d, hookrightarrow]& \ZZ^{r-1} \ar[d] & \\
0 \ar[r] & M \ar[r] \ar[d, two heads] & \ZZ^{\sigma(1)} \ar[r, "A^{\pm}"] \ar[d, two heads]& \ZZ^{r-1} \ar[d, two heads] \ar[r] & 0\\
0 \ar[r] & M_{\tau} \ar[r]  & \ZZ^{\tau(1)} \ar[r, "A^{\pm}_{\tau}"] & \ZZ^{r_{\tau}-1} \ar[r] & 0.
\end{tikzcd}
};
\end{tikzpicture}
\end{equation*}
 
 In particular, the resulting map $A^{\pm}_{\tau}$ is the boundary map of a quiver $Q(\tau)$ obtained by contracting the edges corresponding to $\tau(1)^{c}$. Up to loops, $Q(\tau)$ is isomorphic to $Q_{\underline{V}}$ for some partition $\underline{V}$ of $V$. Therefore, $X(Q^{\pm},0)$ is covered by toric charts of the form
\begin{equation}\label{eq:chart}
    X(Q(\tau), 0) \times N(\tau)\otimes \CC^* \simeq X(Q^{\pm}_{\underline{V}}, 0) \times \text{smooth factor},
\end{equation}
where the isomorphism follows from \cref{lem:loops}. We have obtained the stratification \eqref{eq:whitneyI} which is Whitney by \cref{lem:analyticallytrivialWhitney}. By inspecting the restriction of the equations \eqref{eq:circ} to each chart \eqref{eq:chart}, we get the Whitney stratification \eqref{eq:whitneyII} too; see also \cite[\S 2]{PW2007}.
\end{proof}

\begin{rmk}
     A face $\tau$ of the cone defining the toric variety $X(Q^{\pm}, 0)$ corresponds to a toric orbit $X_{\tau}$ of $X(Q^{\pm}, 0)$. By the proof of \cref{prop:toricLawrenceWhitney} and adopting the notation in there,  $X_{\underline{V}}$ is the union of the toric strata $X_{\tau}$ whose corresponding quiver $Q(\tau)$ is isomorphic to $Q_{\underline{V}}$ up to loops.
\end{rmk}

\begin{prop}
The map $\pi_X: X(Q^{\pm}, \theta) \to X(Q^{\pm}, 0)$ is small.
\end{prop}
\begin{proof}
The fan of $X(Q^{\pm}, \theta)$ is a regular triangulation of the cone $\sigma \subset N \otimes \RR$ of $X(Q^{\pm}, 0)$; see \cite[\S 2]{HauselStrumfels2002}. Since any regular  triangulation of $\sigma$ restricts to a regular triangulation of its face $\tau$, the restriction $\pi_X$ to the open set $X(Q^{\pm}_{\underline{V}}, 0) \times X_{\underline{V}}$ is the resolution
\[X(Q^{\pm}_{\underline{V}}, [\theta]) \times X_{\underline{V}} \to X(Q^{\pm}_{\underline{V}}, 0) \times X_{\underline{V}}.\] Therefore, given $y \in X_{\underline{V}} \subset X(Q^{\pm}, 0) \times X_{\underline{V}}$ we obtain \[
2\dim \pi_X^{-1}(y)= 2 \dim C(Q^{\pm}_{\underline{V}}, [\theta]) =2 b_1(Q_{\underline{V}}) < b_1(Q_{\underline{V}})+s = \dim X(Q^{\pm}_{\underline{V}}, 0) =  \dim X(Q^{\pm}, 0) - \dim X_{\underline{V}},
\]
i.e. $\pi_X$ is small.
\end{proof}

\begin{prop}
The map $\pi_Y: Y(Q, \theta) \to Y(Q, 0)$ is semismall.
\end{prop}
\begin{proof}
By \cref{FactToricQuiver}.(2), $Y(Q, \theta)$ is a smooth symplectic variety. Any symplectic resolution of singularities is semismall by \cite[Lemma 2.11]{Kaledin06}.
\end{proof}


\begin{thm}[Local Hodge-to-singular correspondence]\label{prop:perverselocal}  We have
a canonical isomorphism in $D^bMHM_{\text{alg}}(Y)$ (resp. $D^b(Y, \QQ)$) 
\[
\IC(X(Q^{\pm},0), \QQ)|_{Y(Q,0)}[2b_1]
\simeq  \bigoplus_{\underline{V} \vdash V}
\IC(Y_{\underline{V}},\QQ) \otimes H^{b_1(Q_{\underline{V}})}(Y(Q_{\underline{V}},\theta), \QQ)[2b_1] \langle b_1(Q_{\underline{V}}) \rangle.
\]
In particular, the restriction $\IC(X(Q^{\pm},0), \QQ)|_{Y(Q,0)}[2b_1]$ is pure and perverse.

\end{thm}
\begin{proof} Consider the cartesian square
\[
\begin{tikzpicture}[baseline= (a).base]
\node[scale=0.97] (a) at (0,0){
\begin{tikzcd}
  Y(Q, \theta) \arrow[hookrightarrow, "\iota_{M, \theta}"]{r}\ar[d,"\pi_Y"'] & X(Q^{\pm}, \theta) \ar[d, "\pi_X"] \\
     Y(Q, 0) \arrow[hookrightarrow, "\iota_A"]{r}& X(Q^{\pm},0).
\end{tikzcd}
};
\end{tikzpicture}
\]
The map $\pi_X$ is small, the map $\pi_Y$ is semismall and the inclusion $\iota_{M, \theta}$ is a regular embedding. Therefore,
the decomposition theorem for small maps gives
$\IC(X(Q^{\pm},0), \QQ) \simeq R\pi_{X,*}\QQ_{X(Q^{\pm}, \theta)}$. By proper base change, we obtain
$R\pi_{X,*}\QQ_{X(Q^{\pm}, \theta)}|_{Y(Q,0)}\simeq R\pi_{Y,*}\QQ_{Y(Q, \theta)}$. The statement now follows from the decomposition theorem for semismall maps. In this case, the local systems $\mathrm{R}^{\dim X - \dim Y_k}f_*\QQ_{f^{-1}(Y_k)}$ of \cref{DecThm}
are trivial, supported on the relevant strata $Y_{\underline{V}}$ and with fiber $H^{\mathrm{top}}(\pi_Y^{-1}(y), \QQ) \simeq H^{b_1(Q_{\underline{V}})}(Y(Q_{\underline{V}},\theta), \QQ)$ for all $y \in Y_{\underline{V}}$; see \cite[\S 5]{PW2007}.
\end{proof}

\section{Universal compactified Jacobians}\label{sec:univcompactifiedJac}

The goal of this section is to enlarge the Hitchin fibration to a universal family of compactified Jacobians. Fix integers $g>1$, $n>1$, $d$ as above.  Let $X$ be a projective curve with ample canonical bundle $\omega_X$ and planar singularities. 
\begin{defn}\label{def:compactifiedJac}
 A rank 1 torsion free sheaf $\mathcal{I}$ on $X$ is semistable with respect to the canonical polarization $\omega_{X}$ if for any proper subcurve $Y \subset X$ we have that
 \begin{equation}\label{eq:slopestability}
     \frac{\chi(X, \mathcal{I})}{\deg(\omega_X)}\leq \frac{\chi(Y,\mathcal{I}_{Y})}{\deg(\omega_X|_Y)},
 \end{equation}
 where $\mathcal{I}_Y$ is the biggest torsion free quotient of the restriction $\mathcal{I}|_Y$ of $\mathcal{I}$ to the subcurve $Y$. 
 The \textbf{compactified Jacobian}  $\overline{J}_d(X)$ is 
the projective coarse moduli space of semistable rank 1 torsion free sheaves on $X$ of degree $d+(n-n^2)(1-g)$ with respect to the canonical polarization; see \cite{Simpson1994I}. 
\end{defn}

The BNR correspondence presents the Dolbeault moduli space $M(n,d)$ as the relative compactified
Jacobian of the spectral curve family; see \cite{Hitchin1987a, BNR89, Schaub1998}. It says that given a rank $1$ torsion free sheaf $\mathcal{I}$ on the spectral curve $\pi_a \colon C_a \to C$, the sheaf $\mathcal{E} \simeq \pi_{a,*} \mathcal{I}$ is a vector
bundle of rank $n$ equipped with a Higgs
field $\phi$, induced from the $\mathcal{O}_{C_a}$-module structure. Viceversa, any Higgs bundle $(\mathcal{E}, \phi)$ is the pushfoward of some rank $1$ torsion free sheaf on $C_a$.

\begin{prop} \cite{Hitchin1987a, BNR89, Schaub1998, MRV2019}
The fibre of the Hitchin fibrations $\chi^{-1}(n,d)(a)$ with $a \in A^{\text{red}}_n$ is isomorphic to the compactified Jacobian $\overline{J}_d(C_{a})$ of the spectral curve $C_{a}$.
\end{prop}

\begin{rmk}
The stability condition for rank 1 torsion free sheaves on reduced curves used in \cite[(10.5)]{MRV2019} and \cite{deCataldoHeinlothMigliorini19} coincides with \eqref{eq:slopestability}.
Suppose that $X$ is a reduced spectral curve $C_{a}=\bigcup^r_{i=1} C_i$ whose component $C_i$ has degree $n_i$, and $n=\sum^{r}_{i=1}n_i$. Let $\mathcal{I}$ be a rank 1 torsion free sheaf of degree $d+(n-n^2)(1-g)$ on $C_{a}$. By Riemann--Roch, \eqref{eq:slopestability} reads as follows: for any subcurve $C_{I}=\bigcup_{i \in I} C_i$, with $I \subset \{1, \ldots, r\}$, we have that
\[
     \chi(Y,\mathcal{I}_{Y})\geq (d+n(1-g))\frac{\sum_{i \in I}n_i}{n},
\]
which is the stability condition \cite[(10.5)]{MRV2019}.
\end{rmk}

Following Simpson, Caporaso, Pandharipande and Esteves \cite{Simpson1994I, P1996, E2001} (cf also \cite[\S 2.3, Fact 2.7, Fact 2.12]{CMKV2015} and \cite{deCataldoHeinlothMigliorini19}) we can enlarge the Hitchin fibration to a relative compactified Jacobian of a versal deformation of $C_{a}$. To this end, let $\overline{\mathcal{M}}_{g'}$ be the Deligne--Mumford stack of semi-stable curves of genus $g'\coloneqq n^2(g-1)+1$, and $A^{\times}_n \coloneqq \bigsqcup_{\underline{n}\vdash \, n} S^{\times}_{\underline{n}}$. By \cite[Lem. 2.1]{deCataldoHeinlothMigliorini19} the flat universal spectral curve $\mathcal{C}^{\times}_{n} \to A^\times_{n}$ induces a morphism
\[f^\times: A^{\times}_{n}\to \overline{\mathcal{M}}_{g'}.\]
\begin{prop}[Universal compactified Jacobian]
There exists an open substack $\mathcal{B} \subseteq \overline{\mathcal{M}}_{g'}$
containing $f^{\times}(A^{\times}_{n})$, and
an irreducible (non necessarily smooth) Deligne--Mumford stack $\pi: \unJB \to \mathcal{B}$ that \'{e}tale locally is a
relative compactified Jacobian parametrizing semistable rank 1 torsion free sheaves of degree $d + (n-n^2)(1-g)$ with respect to the canonical polarization.
\end{prop}

\begin{proof}  
The first paragraph of the proof of \cite[Prop. 5.9]{deCataldoHeinlothMigliorini19} does not require the coprimality of $n$ and $d$, and works verbatim in our context.  
\end{proof}
\begin{prop}\label{prop:unramified}
For every partition $\underline{n}$ of $n$, let  $a \in S^{\times}_{\underline{n}}$.
There exists a local multisection $A_{a}$ of $f^{\times}: A^{\times}_{n}\to f^{\times}(A^{\times}_{n}) \subseteq \mathcal{B}$ passing through $a$, i.e. $A_{a}$ is a smooth locally closed subset of $A^{\times}_n$, passing through $a$, 
such that $f^{\times}(A_{a})$ is an open subset in $f^{\times}(A^{\times}_{n})$, and the restriction $f^{\times}|_{A_{a}}: A_{a}\to f^{\times}(A_{a})$ is \'{e}tale.
Furthermore we have a cartesian diagram
\begin{equation*}
\begin{tikzpicture}[baseline= (a).base]
\node[scale=1] (a) at (0,0){
\begin{tikzcd}
\chi(n,d)^{-1}(A_a) \ar[r, "\simeq"] \ar[dr, "{\chi(n,d)}"'] & \unJB \times_{\mathcal{B}}A_{a}  \ar[r, "i_{M}"]\ar[d] & \unJB \ar[d, "\pi"] \\
    & A_a \ar[r, "f^{\times}|_{A_{a}}"] & \mathcal{B}.
\end{tikzcd}
};
\end{tikzpicture}
\end{equation*}
\end{prop}
\begin{proof}
\cite[Cor. 5.10]{deCataldoHeinlothMigliorini19} works verbatim.
\end{proof}
\begin{rmk}\label{rmk:Aa}
The $\Gm$-action on $A_n$ and the translation action $H^0(C,\omega_C)\otimes A_n \to A_n$
lift to the universal spectral curve $\mathcal{C}_{n} \to A_n$ and so to $\MDol(n,d)$. Therefore, the group action induces an analytic isomorphism between an open neighbourhood $U_a \subseteq A_n$ of $A_{a}$ and a neighbourhood of the zero-section of trivial vector bundle $p: A_{a} \times (H^0(C, \mathcal{O}_C) \oplus H^0(C,\omega_C)) \to A_{a}$, such that there exists a homeomorphism
\[\chi(n,d)^{-1}(U_{a})\simeq \chi(n,d)^{-1}(A_{a}) \times (H^0(C, \mathcal{O}_C) \oplus H^0(C,\omega_C)),\]
and so
\[R\chi(n,d)_*\IC(\MDol(n,d), \QQ)|_{U_a}\simeq p^*(R\chi(n,d)_*\IC(\chi(n,d)^{-1}(A_a), \QQ)).\]
\end{rmk}

\section{Kodaira--Spencer map for spectral curves}\label{sec:KS}
  The differential of the map $f^{\times} \colon A_n \to \mathcal{B}$ in \cref{prop:unramified} can be identified with the Kodaira--Spencer map; see for instance \cite[\S 5.2]{deCataldoHeinlothMigliorini19}. Its kernel is isomorphic to $H^0(C, \mathcal{O}_C) \oplus H^0(C, \omega_{C})$; see again \cite[\S 5.2]{deCataldoHeinlothMigliorini19}. In \cref{prop:KS} we study its image.
  
  \subsection{Kodaira--Spencer map} We start with some recollection about the deformation of nodal spectral curves.
\begin{defn} The Kodaira--Spencer map is the map
 \begin{equation}KS_a \colon T_a A_n \simeq H^0(C_a, N_{C_a/T^*C_a}) \to \mathrm{Ext}^1(\Omega^1_{C_a}, \mathcal{O}_{C_a})\end{equation}
 obtained by applying the functor $\mathrm{Hom}(-, \mathcal{O}_{C_a})$ to the relative cotangent sequence
\begin{equation}\label{eq:relativecotangentsequence}
    0 \to N^{*}_{C_a/T^*C} \to \Omega^1_{T^*C}|_{C_a} \to \Omega^1_{C_a} \to 0.
\end{equation}
\end{defn}
Applying the functor $\mathcal{H}om(-, \mathcal{O}_{C_a})$ to \eqref{eq:relativecotangentsequence}, we obtain the exact sequence of coherent sheaves
\[0 \to T_{C_a} \to T_{T^*C}|_{C_a} \to N_{C_a/T^*C} \to \mathbb{T}^1 \coloneqq \mathcal{E}xt^1(\Omega^1_{C_a}, \mathcal{O}_{C_a}) \to 0.\]

The first tangent sheaf $\mathbb{T}^1$ is a skyscraper sheaf supported at the nodes $E(C_{a})$ of $C_a$ and it parametrises smoothings of the nodes. The kernel of $N_{C_a/T^*C} \to \mathbb{T}^1$ is easily identified to $N_{C_a/T^*C} \otimes \mathcal{I}_{E(C_a)}$, where $\mathcal{I}_{E(C_a)}$ is the product of the maximal ideals of the nodes $E(C_{a})$; see \cite[Ex. 4.7.1]{Sernesi2006}. Hence, there exists a short exact sequence 
\[
    0 \to N_{C_a/T^*C} \otimes \mathcal{I}_{E(C_a)} \to N_{C_a/T^*C} \to \mathbb{T}^1 \to 0
\]
inducing the exact sequence in cohomology
\[0 \to H^0(C_a, N_{C_a/T^*C} \otimes \mathcal{I}_{E(C_a)}) \to T_{a}A_{n} = H^0(C_a,N_{C_a/T^*C}) \to \mathbb{T}^1(C_a).
\]
Note that the last map $T_{a}A_{n} \to \mathbb{T}^1(C_a)$ factors as
\begin{equation}\label{eq:Kscomp}
  T_a A_{n}= H^0(C_a,N_{C_a/T^*C}) \xrightarrow{KS_a}  \mathrm{Ext}^1(\Omega^1_{C_a}, \mathcal{O}_{C_a}) \to H^0(C_a, \mathcal{E}xt^1(\Omega^1_{C_a}, \mathcal{O}_{C_a}))=\mathbb{T}^1(C_a),  
\end{equation}
where the latter map comes from the local-to-global Ext spectral sequence.

\subsection{Image of the Kodaira--Spencer map for spectral curves} 
 Let $\underline{n}'$  be a partition of $n$. Then $\chi(n,d)^{-1}(S^{\times}_{\underline{n}'})$ admits a stratification by type 
\[\chi(n,d)^{-1}(S^{\times}_{\underline{n}' }) = \bigsqcup_{\underline{n}\vdash \, n} \chi(n,d)^{-1}(S^{\times}_{\underline{n}'}) \cap M^{\circ}_{\underline{n}}(d).\] 
Given the partition $\underline{n}=(n_1, \ldots, n_r)$, the condition $(\mathcal{E}, \phi) \in \chi(n,d)^{-1}(S^{\times}_{\underline{n}'}) \cap M^{\circ}_{\underline{n}}(d)$ means that the Higgs bundle $(\mathcal{E}, \phi)$ splits as 
\begin{equation}\label{eq:directsumHiggs}
    (\mathcal{E}, \phi) = \bigoplus^r_{i=1} (\mathcal{E}_i, \phi_i),
\end{equation}
where $(\mathcal{E}_i, \phi_i)$ are stable Higgs bundles of rank $n_i$ and degree $d_i = dn_i/n$. 
 Further, the spectral curve $C_{a}$ of $(\mathcal{E}, \phi)$ is nodal and the union of the (possibly reducible) spectral curves $C_{a_i}$ of $(\mathcal{E}_i, \phi_i)$
 \begin{equation}\label{eq:split}
     C_{a}=\bigcup^r_{i=1} C_{a_i}.
 \end{equation}
 
A normal slice $N_{\underline{n}} \subseteq A_{a}$ to $S_{\underline{n}}$ through $a$ parametrises spectral curves obtained by smoothing the nodes $E(\Gamma_{\underline{n}}) \coloneqq \bigcup^r_{i,j=1} C_{a_i} \cap C_{a_j}$ of $C_a$.

\begin{notation}
For a partition $\underline{n}$ of $n$ we will denote by $\Gn$ the dual graph of any spectral
curve given by a point of $S^{\times}_{\underline{n}}$, i.e.\ $\Gn$ is the graph with vertices $[r]\coloneqq \{1, \ldots , r\}$ corresponding to
the irreducible components of the curve and $n_in_j (2g-2)$ edges between the vertices $i$, $j$,
corresponding to the intersection points of the components. By abuse of
notation we continue to denote by $E(\Gamma_{\underline{n}})$ the set of edges of $\Gamma_{\underline{n}}$.
\end{notation} 

\begin{prop}\label{prop:KS}
The image of the map
\[T_a N_{\underline{n}} \hookrightarrow T_a A_{a} \xrightarrow{df^{\times}}  \mathrm{Ext}^1(\Omega^1_{C_a}, \mathcal{O}_{C_a}) \twoheadrightarrow \bigoplus_{e \in E(\Gamma_{\underline{n}})}H^0(C_a, \mathcal{E}xt^1(\Omega^1_{C_a}, \mathcal{O}_{C_a})_e) \simeq \Spec \CC[T_e \colon e \in E(\Gamma_{\underline{n}})]\]
is cut by the linear equations
\[ \mathrm{Circ}(\mathscr{B})=\langle \sum_{e \in E} a_{ie}T_e \colon i=2, \ldots r\rangle, \]
where $A=(a_{ie})$ is the boundary map of a quiver underlying the graph $\Gamma_{\underline{n}}$ (cf Section \ref{sec:toric hyperkahler variety}).
\end{prop}
\begin{proof}
Let $\nu_{\underline{n}} \colon C _{\underline{n}, a} \coloneqq \bigsqcup^r_{i=1}C_{a_i} \to C_{a}$ be the partial normalization obtained by blowing-up $C_{a}$ at $E(\Gamma_{\underline{n}})$. Applying the functor $\mathcal{H}om( - , \mathcal{\omega}_{C_a})$ to the short exact sequence
\[0 \to \mathcal{O}_{C_a} \to \nu_{\underline{n}, *}\mathcal{O}_{C _{\underline{n}, a}} \to \bigoplus_{e \in E(\Gamma_{\underline{n}})} \mathcal{O}_e \to 0,\]
we obtain the short exact sequence  
\[
    0 \to N_{C_a/T^*C} \otimes \mathcal{I}_{E(\Gamma_{\underline{n}})} \to N_{C_a/T^*C} \to \mathbb{T}^1_{\underline{n}} \coloneqq \bigoplus_{e \in E(\Gamma_{\underline{n}})} \mathcal{E}xt^1(\mathcal{O}_e, \omega_{C_a}) \to 0,
\]
where $\mathcal{I}_{E(\Gamma_{\underline{n}})}$ is the product of the maximal ideals of the nodes $E(\Gamma_{\underline{n}})$, and $\omega_{C_a}$ is identified with $N_{C_a/T^*C}$ by adjunction. Taking global cohomology, we get the long exact sequence 
\begin{align} 0 \to H^0(C_a, N_{C_a/T^*C} \otimes \mathcal{I}_{E(\Gamma_{\underline{n}})}) \to T_{a}A_{n} = H^0(C_a,N_{C_a/T^*C}) \to \mathbb{T}^1_{\underline{n}}(C_a) \to  \nonumber\\
\to H^1(C_a, N_{C_a/T^*C} \otimes \mathcal{I}_{E(\Gamma_{\underline{n}})}) \to H^1(C_a,N_{C_a/T^*C}) \to 0. \label{eq:normcomb}
\end{align}
By Serre duality we see
\begin{align} \label{eq:TAN2}
 H^0(C_a, N_{C_a/T^*C} \otimes \mathcal{I}_{E(\Gamma_{\underline{n}})}) \simeq & H^0(C_a,\mathcal{H}om( \nu_{\underline{n}, *} \mathcal{O}_{C _{\underline{n}, a}}, \omega_{C_a})) \simeq H^1(C_a, \nu_{\underline{n}, *} \mathcal{O}_{C _{\underline{n}, a}})^{\vee} \nonumber\\
    & \simeq \bigoplus^r_{i=1} H^1(C_i, \mathcal{O}_{C_i})^{\vee} \simeq \bigoplus^r_{i=1} H^0(C_i, \omega_{C_i})\simeq T_a S_{\underline{n}}.
\end{align}
Combining \eqref{eq:Kscomp}, \eqref{eq:normcomb} and \eqref{eq:TAN2}, the image of the map
\[df^{\times} \colon T_a N_{\underline{n}} \simeq T_a A_{n}/T_a S_{\underline{n}} \to \mathbb{T}^1_{\underline{n}}(C_a)\]
coincides with the kernel of $\mathbb{T}^1_{\underline{n}}(C_a) \to \ker\{H^1(C_a, N_{C_a/T^*C} \otimes \mathcal{I}_{E(\Gamma_{\underline{n}})}) \to H^1(C_a,N_{C_a/T^*C})\}$. Again via Serre duality, this map is isomorphic to 
\[\CC^{E(\Gamma_{\underline{n}})} \simeq \bigoplus_{e \in E(\Gamma_{\underline{n}})} H^0(\mathcal{O}_e)^{\vee} \to \ker\{H^0(C_a, \nu_{\underline{n}, *}\mathcal{O}_{C _{\underline{n}, a}} )^{\vee} \to H^0(C_a, \mathcal{O}_{C_a})^{\vee}\}\simeq \CC^{r-1},\]
which in turn can be identified with the boundary map of a quiver underlying the graph $\Gamma_{\underline{n}}$.
\end{proof}

\section{Singularities of universal compactified Jacobians}\label{sec:singularitiesUnivComp}

By the BNR correspondence a Higgs bundle $(\mathcal{E}, \phi)$ corresponds to a rank 1 torsion free sheaf $\mathcal{I}$ on the spectral curve $C_{a}$, where $\mathcal{E} $ is isomorphic to $\pi_{a,*} \mathcal{I}$ via the projection $\pi_a\colon C_{a} \to C$; see \cite{Hitchin1987a, BNR89, Schaub1998}. Suppose that $C_{a}$ is nodal, and let $\Sigma$ be the set of the nodes where $\mathcal{I}$ fails to be locally free. Define $M^{\times}_{\underline{n}}(d)$ as the open subsets of Higgs bundles of $M^{\circ}_{\underline{n}}(d)$ with nodal spectral curve.

\begin{prop}\label{prop:charMn}
Let $\underline{n}=(n_1, \ldots, n_r)$ be a partition of $n$. Then $(\mathcal{E}, \phi)$ lies in $M^{\times}_{\underline{n}}(d)$ if and only if there exists a partial normalization $\nu_{\Sigma} \colon C_{\Sigma} \to C_{a}$ obtained by blowing-up the nodes $\Sigma$ and a line bundle $\mathcal{L}$ on $C_{\Sigma}$ such that (1) $C_{\Sigma}$ has exactly $r$ connected components, and (2) $\mathcal{E} \simeq \nu_{\Sigma, *}\mathcal{L}$.
\end{prop}

\begin{proof}
By assumption, $(\mathcal{E}, \phi)$ splits as $(\mathcal{E}, \phi) = \bigoplus^r_{i=1} (\mathcal{E}_i, \phi_i)$,
where $(\mathcal{E}_i, \phi_i)$ are stable Higgs bundles of rank $n_i$ and degree $d_i = dn_i/n$. The spectral curve $C_{a}$ of $(\mathcal{E}, \phi)$ is the union of the spectral curves $C_{a_i}$ 
\[C_{a_i} \xhookrightarrow{j_i}C_{a}= \bigcup^r_{i=1}C_{a_i} \xrightarrow{\pi_a} C.\]
By the BNR correspondence, there exist rank $1$ torsion free sheaves $\mathcal{I}$ and $\mathcal{I}_i$  on $C_{a}$ and $C_{a_i}$ respectively such that $\mathcal{E} = \pi_{a, *} \mathcal{I}$, $\mathcal{E}_i = (\pi_{a} \circ j_i)_{*} \mathcal{I}_i$, and 
\begin{equation}\label{eq:rankonetorsion}
    \mathcal{I} = \bigoplus^r_{i=1} j_{i, *}\mathcal{I}_i.
\end{equation}
Note that $\mathcal{I}$ fails to be locally free at the node $C_{a_i} \cap C_{a_j}$, $i \neq j$. Therefore, the blow-up $\nu_{\Sigma} \colon C_{\Sigma} \to C_{a}$ of $C_{a}$ at $\Sigma$ factors through the disjoint union of $C_{a_i}$
\begin{equation}\label{eq:partialnorml}
    \nu_{\Sigma} \colon C_{\Sigma} \to \bigsqcup^r_{i=1}C_{a_i} \to C_{a}.
\end{equation}
Since $C_{a}$ has only nodal singularities, there exists a line bundle $\mathcal{L}$ on $C_{\Sigma}$ such that $\mathcal{I} = \nu_{\Sigma, *} \mathcal{L}$; see \cite[Prop. 3.4]{Gagne97}. By \eqref{eq:partialnorml}, the curve $C_{\Sigma}$ has at least $r$ components, and exactly $r$ by the stability of $(\mathcal{E}_i, \phi_i)$. Otherwise, the decomposition \eqref{eq:rankonetorsion} could be further refined, and so \eqref{eq:directsumHiggs} could be refined too, which is a contradiction. 
\end{proof}

\begin{prop}\label{prop:dualcomplex}
Suppose that $(\mathcal{E}, \phi)$ lies in $M^{\times}_{\underline{n}}(d)$.
After removing all the loops, the dual graph $\Gamma_{C_{a}}(\Sigma)$ of the curve obtained by smoothing the nodes not in $\Sigma$ is the graph $\Gn$.
\end{prop}

\begin{proof} As in the proof of \cref{prop:charMn}, we denote the spectral curves of $(\mathcal{E}, \phi)$ and of its stable summands by $C_{a}$ and $C_{a_i}$ respectively, and we write
\[C_{a} = \bigcup^r_{i=1} C_{a_i}=\bigcup^{r'}_{j=1}C_j \qquad C_{a_i}=\bigcup_{j\in S_i}C_j,\]
where $C_j$ are the irreducible components of $C_{a}$, and $[r']=\bigsqcup^r_{i=1} S_{i}$ is a partition of the set of $r'$ elements. Clearly if $C_j$ has degree $n'_j$ over $C$, the spectral curves $C_{a_i}$ has degree $n_i = \sum_{j \in S_i}n'_j$.

The graph $\Gamma_{C_{a}}(\Sigma)$ is obtained from the dual graph $\Gamma_{C_{a}}$ of the spectral curve $C_{a}$ by contracting all the edges not in $\Sigma$.\footnote{By abuse of notation, we identify the set of nodes $\Sigma$ with the corresponding set of edges in the dual graph $\Gamma_{C_{a}}$.}
Let $q \colon \Gamma_{C_{a}} \to \Gamma_{C_{a}}(\Sigma)$ be the contraction map. If some edge between $v_{j}$ and $v_{k}$ is not contained in $\Sigma$, then $q(v_{j})=q(v_{k})$, and all edges in $\Sigma$ are sent to loops in $\Gamma_{C_{a}}(\Sigma)$. The graph $\Gamma_{C_{a}}(\Sigma)$ has as many vertices as connected components in $\Gamma_{C_{a}}\setminus \Sigma \simeq \Gamma_{C_{\Sigma}}$, i.e. $r$ by \cref{prop:charMn}. The edges between two distinct vertices $w_{i}$ and $w_{i'}$ in $\Gamma_{C_{a}}(\Sigma)$ are exactly 
\begin{align*}
  \sum_{\substack{v_{j} \in q^{-1}(w_{i}), \\v_{k} \in q^{-1}(w_{i'})}} (\# \text{ edges between }v_{j}\text{ and }v_k) & = \sum_{\substack{v_{j} \in q^{-1}(w_{i}), \\v_{k} \in q^{-1}(w_{i'})}} n'_j n'_k(2g-2)\\
  & =\bigg(\sum_{j \in S_{i}}n'_j\bigg) \cdot \bigg(\sum_{k \in S_{i'}}n'_k\bigg)(2g-2)= n_i n_{i'} (2g-2).
\end{align*}
We conclude that the graphs $\Gamma_{C_{a}}(\Sigma)$ and $\Gn$ coincide after removing all loops in $\Gamma_{C_{a}}(\Sigma)$. 
\end{proof}

\begin{rmk}
The dual graph of the spectral curve may depend on the choice of $(\mathcal{E}, \phi)$ in $M^{\times}_{\underline{n}}(d)$. \cref{prop:dualcomplex} asserts that up to loops the dual graph $\Gamma_{C_{a}}(\Sigma)$ does not! 
\end{rmk}


The singularities of the universal compactified Jacobian $\unJB$ are affine Lawrence varieties $X(\Gamma_{\underline{n}}^{\pm}, 0)$, as explained in \cite{CMKV2015} and \cite{CMKV2017}. The corresponding hypertoric quiver variety  $Y(\Gamma_{\underline{n}}, 0)$ provides the local model for the singularities of $\MDol(n,d)$.

\begin{thm}[Local model of $\unJB$ and $\MDol(n,d)$]\label{thm:localmodelJM} Let  $a \in S^{\times}_{\underline{n}'}$, and $A_{a}$ be the local multisection of $f^{\times}: A^{\times}_{n}\to f^{\times}(A^{\times}_{n}) \subseteq \mathcal{B}$ defined in \cref{prop:unramified}.
Let $(\mathcal{E}, \phi)$ be a Higgs bundle in $ M^{\times}_{\underline{n}}(d)$,
and $\mathcal{I}$ be the corresponding rank 1 torsion free sheaf on the spectral curve $C_{a}$. 
Then there exists an analytic neighbourhood $V$   of an \'{e}tale chart of $\unJB$ centered at $(C_a, \mathcal{I})$   such that the restriction to $V$ of the fibre product square in \cref{prop:unramified}
\begin{equation*}
\begin{tikzpicture}[baseline= (a).base]
\node[scale=1] (a) at (0,0){
\begin{tikzcd}
\MDol(n,d) \supset U \ar[r, hookrightarrow]\ar[two heads]{d}[swap]{\chi(n,d)} & V \subset \unJB \ar[d, two heads, "\pi"] \\
     \chi(n,d)(U)\ar[r, hookrightarrow, "f^{\times}"] & \pi(V),
\end{tikzcd}
};
\end{tikzpicture}
\end{equation*}
with $U \coloneqq V \cap \chi(n,d)^{-1}(A_a)$,
is locally isomorphic to
\begin{equation}
\begin{tikzpicture}[baseline= (a).base]
\node[scale=1] (a) at (0,0){
\begin{tikzcd}
  Y(\Gamma_{\underline{n}}, 0) \times \CC^{\dim M_{\underline{n}}-1-g}   \arrow[hookrightarrow ]{rr}{(\iota_{M}, l_1)}\ar[two heads]{d}[swap]{(\chi_{\Gamma_{\underline{n}}}, l_2)} & & X(\Gamma_{\underline{n}}^{\pm}, 0)\times \CC^{c(\underline{n}, g) } \ar[two heads]{d}{(\pi_{\Gamma_{\underline{n}}}, l_4)} \\
     \CC^{b_1(\Gamma_{\underline{n}})} \times  \CC^{d(\underline{n},g)}  \arrow[hookrightarrow ]{rr}{(\iota_A,l_3)}& & \CC^{s} \times  \CC^{3n^2(g-1)-s},
\end{tikzcd}
};
\end{tikzpicture}
\end{equation}
where the maps $\chi$, $\pi$, $\iota_M$ and $\iota_A$ are defined in \eqref{eq:diagramsingular}, the maps $l_i$ are linear, $s \coloneqq \#E(\Gamma_{\underline{n}})$, $d(\underline{n},g) \coloneqq \frac{1}{2}(\dim \MDol(n,d)-\dim Y(\Gamma_{\underline{n}}, 0)) -g-1 = (n^2-1)(g-1)-1 -b_1(\Gamma_{\underline{n}})$ and $c(\underline{n}, g) \coloneqq \dim \unJB -\dim X(\Gamma_{\underline{n}}^{\pm}, 0) = 4n^2(g-1)+1 - b_1(\Gamma_{\underline{n}})-s$.
\end{thm}

\begin{proof} 
Let $\Sigma$ be the set of the nodes where $\mathcal{I}$ fails to be locally free, and $\Gamma_{C_{a}}(\Sigma)$ be the dual graph of the curve obtained by smoothing the nodes not in $\Sigma$. 
By \cite[Theorem A]{CMKV2015} there exists an analytic neighbourhoods $V$ of $(C_a, \mathcal{I})$ in $\unJB$ such that \[V \simeq_{\mathrm{loc}} X(\Gamma_{C_{a}}(\Sigma)^{\pm}, 0) \times \CC^{c(\underline{n},\Sigma)},\]
with $c(\underline{n},\Sigma) \coloneqq 4 g(C_a)-3-b_1(\Gamma_{C_{a}}(\Sigma))-\#E(\Gamma_{C_{a}}(\Sigma))$. By \cite[Lemma 4.3]{CMKV2017} and \cref{prop:dualcomplex} we obtain that 
\[V \simeq_{\mathrm{loc}} X(\Gamma_{C_{a}}(\Sigma)^{\pm}, 0) \times \CC^{c(\underline{n},\Sigma)} \simeq X(\Gamma_{\underline{n}}^{\pm}, 0) \times \CC^{c(\underline{n},g)}.\]
The fact that the morphism $\pi \colon V \subset \unJB \to \pi(V)$ can be identified with $(\pi_{\Gamma_{\underline{n}}}, l_4)$ is explained in \cite[\S 7.2]{CMKV2017}. The differential of the map $f^{\times}$ is the Kodaira-Spencer map $KS_a \colon T_a A_a \to H^1(C_a, \mathbb{T}^1_{C_a})$, and so the identification of the Hitchin fibration $\chi(n,d)$ with the product of the map 
$\chi_{\Gamma_{\underline{n}}} \colon Y(\Gamma_{\underline{n}}, 0) \to \CC^{b_1(\Gamma_{\underline{n}})}$ with a linear map $l_2$ follows from \cref{prop:KS} and \cref{def:torichyperkahlervarieties}.
\end{proof}
\begin{thm}\label{prop:perverse} Let $A_a \subset A_{n}$ as in \cref{prop:unramified}. Then  $\IC(\unJB,\QQ)|_{\chi(n,d)^{-1}(A_a)}[\dim M(n,d)]$ is a semi-simple perverse sheaf. In particular, we have
\[
\IC(\unJB,\QQ)|_{\chi(n,d)^{-1}(A_a)}[\dim M(n,d)]\simeq \bigoplus_{\substack{\underline{n} \in \mathcal{P}_d}} i_{M}^*
\IC(M_{\underline{n}}(d),\mathcal{L}'_{\underline{n}})[\dim M(n,d)]\langle \codim S_{\underline{n}} \rangle,
\]
where $\mathcal{L}'_{\underline{n}}\coloneqq \mathcal{H}^{\codim M_{\underline{n}}(d)}(\IC(\unJB,\QQ))$ is a local system on $M_{\underline{n}}(d)$ of rank $h^{b_1(\Gamma_{\underline{n}})}(Y(\Gamma_{\underline{n}},0), \QQ)$.
\end{thm}
We anticipate that the local systems $\mathcal{L}'_{\underline{n}}$ are pullback of the line bundles $\mathcal{L}_{\underline{n}}$ in \cite[Cor. 6.20]{deCataldoHeinlothMigliorini19}; see \cref{cor:trivialityalongfibers} and the proof of \cref{cor:degree1} below. In particular, their monodromy is known and induced by the automorphisms of the graph $\Gamma_{\underline{n}}$; see \cref{rmk:repIH}.

The idea of the proof of \cref{prop:perverse} is to glue the local Hodge-to-singular correspondence (\cref{prop:perverselocal}) given the local models in \cref{thm:localmodelJM}.

\begin{proof} Let $X$ be a complex algebraic variety. The category $MH(X, w)$ of polarisable Hodge modules of fixed weight $w$ admits an exact and faithful functor $\mathrm{rat} \colon MH(X, w) \to \mathrm{Perv}(X)_{\QQ}$ to the category of perverse sheaves (with coefficient in $\QQ$); see \cite{Saito1988}. The category $MH(X, w)$ is semisimple by \cite[Lemma 5]{Saito1988}, and the perverse sheaf that one obtain applying the functor $\mathrm{rat}$ is semisimple \cite[Lemma 4]{Saito1988}. Therefore, 
the semisimplicity of $\IC(\unJB,\QQ)|_{\chi(n,d)^{-1}(A_a)}[\dim M(n,d)]$ follows from purity and perversity. 
Since purity and perversity are local conditions in the analytic topology (cf for instance \cite[Lem. 2.2]{Davison2021}), it is enough to check them on the local model, which  is the content of \cref{prop:perverselocal}.
\end{proof} 

\begin{prop}\label{cor:trivialityalongfibers}
Define $M^{\circ \circ}_{\underline{n}}(d) \coloneqq \chi(n,d)^{-1}(S^{\times}_{\underline{n}}) \cap M_{\underline{n}}(d)$. The local system $\mathcal{L}'_{\underline{n}}$ on $M^{\circ \circ}_{\underline{n}}(d)$ descends to a local system $\mathcal{L}_{\underline{n}}$ on $S^{\times}_{\underline{n}}$
\[\mathcal{L}'_{\underline{n}}=\chi(n,d)^{*}\mathcal{L}_{\underline{n}}.\]
\end{prop}
\begin{proof}
Denote by $F_{b}$ the fibre of $\chi(n,d) \colon M^{\circ \circ}_{\underline{n}}(d) \to S_{\underline{n}}$. By \cref{prop:localtrivial}, the restriction $\IC(\unJB,\QQ)|_{F_{b}}$ is isomorphic to the complex 
\[\QQ_{F_{b}} \otimes \IH^{i}(X(\Gamma_{\underline{n}}^{\pm},0), \QQ)[-i] \simeq \QQ_{F_{b}} \otimes H^{i}(Y(\Gamma_{\underline{n}},0), \QQ)[-i]\]
of constant sheaves with trivial differential. Hence, $\mathcal{L}'_{\underline{n}}|_{F_{b}} \simeq \QQ_{F_{b}} \otimes H^{b_1(\Gamma_{\underline{n}})}(Y(\Gamma_{\underline{n}},0), \QQ)$ is a component of the complex, so a trivial local system. Therefore, $\mathcal{L}'_{\underline{n}}$ descends to a local system on $S^{\times}_{\underline{n}}$.
\end{proof}

\section{Tubular neighbourhood of the Jacobian of the normalisation of a spectral curve}\label{sec:tubularneighbou}
To complete the proof of \cref{cor:trivialityalongfibers}, we must trivialize a tubular neighbourhood of the subvariety $F_{b}$ in the universal compactified Jacobian.   This is a global statement which asserts that the local trivializations in \cref{thm:localmodelJM} glue along $F_{b}$.  

Let $f\colon \mathcal{C} \to B$ be a flat projective versal family of semistable curves and arithmetic genus $g'$, and $\pi\colon \unJvers \to B$ be the relative compactified Jacobian with respect to the canonical polarization. 
For any $b \in B$, we denote by $F_{b}$ the subvariety in $\pi^{-1}(b)$ parametrising rank 1 torsion free sheaves $\mathcal{I}$ on $C_b \coloneqq f^{-1}(b)$ which are pushforward of line bundles from the normalization of $C_b$, i.e.\ $\mathcal{I} = \nu_* M$ where $M$ is a line bundle on the normalization $\nu \colon C^{\nu}_{b}=\bigsqcup^{r}_{i=1} C_{b,i} \to C_{b}$. Denote by $\Sigma$ the set of nodes of $C_b$.

\begin{prop}\label{prop:localtrivial}
There exists a homeomorphism between a euclidean neighbourhood of $F_{b}$ in $\unJvers$ and the product
\begin{equation}\label{eq:product}
    F_{b} \times X(\Gamma(C_b)^{\pm}, 0) \times \CC^{3g'-3-s},
\end{equation}
where $\Gamma(C_b)$ is the dual graph of $C_b$ and $s$ is the number of edges of the graph.
\end{prop}
\begin{rmk}\label{rmk:analvshomeo}
Let $B_{\Gamma} \subseteq B$ the locus of curve with the same dual graph as $C_{b}$. The restriction of the family $\mathcal{C}$ to $B_{\Gamma}$ can be simultaneously normalized, and let $\mathcal{C}^{\nu} \to \mathcal{C} \times_{B} B_{\Gamma}$ the simultaneous normalization.  \cref{prop:localtrivial} says in particular that a euclidean neighbourhood of $F_{b}$ in the relative Picard scheme $\Pic^{d'}(\mathcal{C}^{\nu}) \subseteq \unJvers$ is homeomorphic to $F_{b} \times \CC^{3g'-3-s}$, but this is not a formal or analytic isomorphism. Otherwise $T \Pic^{d'}(\mathcal{C}^{\nu})|_{F_{b}}\simeq \mathcal{O}^{\oplus c'}_{F_b}$ with $c'\coloneqq \dim(F_b)+3g'-3-s$, and the boundary map $H^{0}(N_{F_{b}/\Pic^{d'}(\mathcal{C}^{\nu})}) \to H^{1}(TF_{b})$ of the short exact sequence
\[
0 \to TF_{b} \to T\Pic^{d'}(\mathcal{C}^{\nu})|_{F_{b}}\to N_{F_{b}/\Pic^{d'}(\mathcal{C}^{\nu})}\to 0
\]
would be trivial, i.e.\ $\mathcal{C}^{\nu}$ is an isotrivial family, which is a contradiction. This suggests that \cref{prop:localtrivial} is false in the analytic category. 
\end{rmk}

In order to show \cref{prop:localtrivial}, we need to recall the GIT description of $\unJvers$ and the deformation theory of the pair $(C_{b}, \mathcal{I})$ as in \cite{CMKV2015, CMKV2017}. This construction goes back to Simpson \cite{P1996}. Over the versal curve $f \colon \mathcal{C} \to B$ one considers the relative quot scheme $\mathrm{Quot}(\mathcal{O}^{\oplus R})$, where $R$ is an appropriately high positive integer which can be made explicit, parameterising quotient sheaves $\mathcal{O}^{\oplus R} \to \mathcal F \to 0$ with a fixed Hilbert polynomial. The group $\mathrm{SL}_{R}$ acts on $\mathrm{Quot}(\mathcal{O}^{\oplus R})$ by changing the basis of $\mathcal{O}^{\oplus R}$ by linear maps in $\Sl_R$. The GIT stability for the natural linearized action of $\Sl_R$ on  $\mathrm{Quot}(\mathcal{O}^{\oplus R})$ corresponds to slope stability. Then 
the relative compactified Jacobian $\unJvers$ can then be constructed as the GIT quotient
\[ \unJvers = Q^{ss} \sslash \mathrm{SL}_{R},\]
where $Q^{ss} \subseteq \mathrm{Quot}(\mathcal{O}^{\oplus R})$ is the locus of semistable rank 1 torsion free quotients. See \cite[\S I]{Simpson1994I} for details or \cite[Fact 2.7]{CMKV2015}.
 
Note that $Q^{ss}$ is smooth along the closed $\mathrm{SL}_{R}$-orbits by \cite[Lemma 6.4.(ii)]{CMKV2015} and by the smoothness of the miniversal deformation of the pair $(C_{b}, \mathcal{I})$; see for instance \cite[\S 3]{CMKV2015}.
 
Let $Z$ be the locus of closed orbits over $F_{b} \subset \unJvers = Q^{ss}\sslash \Sl_R$. Note that $F_{b}$ is isomorphic to the Jacobian $\Pic^{d'}(C^{\nu}_{b})$ with $d'\coloneqq d+(n-n^2)(1-g)$, and $Z$ is a fiber bundle over $F_b$. 

Now let
\begin{equation}\label{eq:Q}
    \mathfrak{n}'\colon \mathcal{Q} = \mathrm{Bl}_{Z \times 0}(Q^{ss} \times \CC) \setminus Q^{ss}\times 0\to \CC
\end{equation} 
be the deformation of $Q^{ss}$ to the normal cone of $Z$. The central fibre is the total space $\Tot(N_Z)$ of the normal bundle $N_{Z}$ of $Z$ in $Q^{ss}$, and the restriction of $\mathfrak{n}'$ to $\CC^*$ is the trivial fibration $Q^{ss} \times \CC^* \to \CC^*$. Since $Z$ is $\Sl_R$-invariant, $\mathcal{Q}$ inherits a $\Sl_R$-action preserving the fibres of $\mathfrak{n}'$. The quotient
\begin{equation*}
    \mathfrak{n}\colon \mathcal{Q}\sslash\Sl_R \to \CC
\end{equation*}
is a deformation of $\unJB$ to $\Tot(N_Z)\sslash\Sl_R$. We can then reduce the proof of \cref{prop:localtrivial} to the following statements:
\begin{itemize}
    \item $\Tot(N_Z)\sslash\Sl_R$ is homeomorphic to the product \eqref{eq:product};
    \item $\Tot(N_Z)\sslash\Sl_R$ is homeomorphic to a euclidean neighbourhood of $F_{b}$ in $\unJvers$.
\end{itemize}
We prove the statements respectively in \cref{lem:triviality} and \cref{sec:Ehresmann}.
\begin{rmk}
In general, the deformation $\mathfrak{n}$ is not the deformation of $\unJvers$ to the normal cone of $F_{b}$. Indeed, a neighbourhood $U$ of a point in $F_{b}$ is locally analytically isomorphic to $X(\Gamma(C_b)^{\pm}, 0) \times \CC^{c'}$ with $c'\coloneqq 4g'-3-b_1(\Gamma)-s$ by \cite[Theorem A]{CMKV2015}. Up to a smooth factor, the degeneration to the normal cone deforms $U$ to 
the normal cone of $X(\Gamma(C_b)^{\pm}, 0)$ at $0$. The latter may fail to be isomorphic to $X(\Gamma(C_b)^{\pm}, 0)$: they are isomorphic only if the homogeneous polynomial cutting $X(\Gamma(C_b)^{\pm}, 0)$ have the same degree, equivalently if the number of edges between different vertices of $\Gamma(C_b)^{\pm}$ is constant. On the other hand,  $\mathfrak{n}$ is locally trivial, so it deforms $U$ to $X(\Gamma(C_b)^{\pm}, 0) \times \CC^{c'}$.  
\end{rmk}
\subsection{Deformation theory of rank 1 torsion free sheaves}
A slice of $Q^{ss}$ normal to an $\Sl_{R}$-orbit of the pair $(C_{b},\mathcal{I})$ is given by a miniversal deformation $\defo(C_{b},\mathcal{I})$ of the pair; see \cite[\S 6]{CMKV2015}. 

The space $\defo(C_{b},\mathcal{I})$ is endowed with a natural action of the torus \[\Aut(\mathcal{I}) \simeq H^0(C_{b}^{\nu}, \mathcal{O}^*_{C_{b}^{\nu}})\simeq \bigoplus^r_{i=1} H^0(C_{b,i}, \mathcal{O}^*_{C_{b,i}}) \simeq \mathbb{T}^{r} \ni (\lambda_i)^r_{i=1},\] see \cite[Def. 3.4]{CMKV2015}. Since $\Aut(\mathcal{I})$ is reductive, 
the tangent space $T\defo(C_{b},\mathcal{I})$ splits into the product 
\[T\defo(C_{b},\mathcal{I})_{\mathrm{inv}} \oplus T\defo(C_{b},\mathcal{I})_{\mathrm{var}},\] where the invariant part $T\defo(C_{b},\mathcal{I})_{\mathrm{inv}}$ 
is fixed by the action of $\Aut(\mathcal{I})$, while the variant part $T\defo(C_{b},\mathcal{I})_{\mathrm{var}}$ is the unique $\Aut(\mathcal{I})$-invariant complement of $T\defo(C_{b},\mathcal{I})_{\mathrm{inv}}$. In fact, these subspaces have a modular interpretation: 
\[
    T\defo(C_{b},\mathcal{I})_{\mathrm{inv}} \simeq T\defo^{lt}(C_{b},\mathcal{I}) \qquad  T\defo(C_{b},\mathcal{I})_{\mathrm{var}} \simeq \prod_{e \in \Sigma} T\defo(C_{b, e},\mathcal{I}_e),
\]
where $\defo^{lt}(C_{b},\mathcal{I})$ is the miniversal deformation of locally trivial deformations of $(C_{b},\mathcal{I})$, and $\defo(C_{b, e},\mathcal{I}_e)$ is a miniversal deformation of the localised pair $(C_{b, e},\mathcal{I}_e)$; see \cite[\S 3 and \S 5]{CMKV2015}. 
In particular, it follows from  \cite[\S 5]{CMKV2015} that $\Aut(\mathcal{I})$ acts on $T\defo(C_{b, e},\mathcal{I}_e)\simeq \Spec k[z_e, w_e]$ as 
\begin{equation}\label{eq: groupactionII}
z_e \mapsto \lambda_{s(e)} z_e \lambda_{t(e)}^{-1} \qquad w_e \mapsto \lambda_{s(e)}^{-1} w_e \lambda_{t(e)}.
\end{equation}

Now let $q: Q^{ss} \to \unJvers$ be the quotient map. Choose $\mathcal{U}$ a universal rank 1 torsion-free sheaf over $Q^{ss} \times_{B} \mathcal{C}$, and take $p: Q^{ss}  \times_{B} \mathcal{C} \to {Q}^{ss}$ the natural projection. 
\begin{lem}\label{lem:splitnormalbundle} Let $N_{Z}$ be the normal bundle of $Z$ in $Q^{ss}$.
There exists a $\mathbb{T}^{r}$-equivariant splitting
\[N_Z \simeq N^{\mathbb{T}^{r}}_{Z} \oplus R^0 p_*\mathcal{E}xt^1_{C_b}(\mathcal{U}|_{Z}, \mathcal{U}|_{Z}),\]
where $N^{\mathbb{T}^{r}}_{Z}$ is the $\mathbb{T}^{r}$-invariant part of $N_Z$.
\end{lem}
\begin{proof}
The normal space to $Z$ in $Q^{ss}$ at $z \in Z$ with $q(z)=(C_{b}, \mathcal{I})$ admits the $\mathbb{T}^{r}$-equivariant splitting
\begin{equation}\label{eq:Normal}
    N_{Z,z} \simeq (T_{z}  \defo^{lt}(C_{b},\mathcal{I})/q^*T_{(C_b, \mathcal{I})}F_{b}) \oplus \bigoplus_{e \in \Sigma} T_{z}\defo(C_{b, e},\mathcal{I}_e),
\end{equation}
and $\mathbb{T}^{r}$ acts trivially on the first summand $(T_{z}  \defo^{lt}(C_{b},\mathcal{I})/q^*T_{(C_b, \mathcal{I})}F_{b})$ and as \eqref{eq: groupactionII} on the second summand.
By \cite[Lemma 7.13]{Drezet04} and the local-to-global Ext spectral sequence, $\bigoplus_{e \in \Sigma} T_{z}\defo(C_{b, e},\mathcal{I}_e)$ can be identified with
\[H^0(C_{b},\mathcal{E}xt^1_{C_{b}}(\mathcal{I}, \mathcal{I}))\simeq \bigoplus_{e \in \Sigma}H^0(C_b, \mathcal{E}xt^1(\mathcal{I}, \mathcal{I})_e) \simeq \CC^{s},\]
i.e.\ the fibre of the vector bundle $R^0 p_*\mathcal{E}xt^1_{C_b}(\mathcal{U}|_{Z}, \mathcal{U}|_{Z})$.
\end{proof}

\begin{lem}\label{lem:trivialvectorbundle}
$R^0 p_*\mathcal{E}xt^1_{C_b}(\mathcal{U}|_{Z}, \mathcal{U}|_{Z})$ is a trivial vector bundle on $Z$.
\end{lem}
\begin{proof}
The group $F_{b} \simeq \Pic^{d'}(C_b^{\nu})$ acts on $Z$. Its action lifts to the vector bundle $R^0 p_*\mathcal{E}xt^1_{C_b}(\mathcal{U}|_{Z}, \mathcal{U}|_{Z})$ parallelising it, as we show below. 

A point in $Z$ represents a quotient $\mathcal{O}_{C_b}^{\otimes R} \to \nu_*M$, where $M$ is a line bundle on $C_b^{\nu}$.
Note that $\ext^1_{C_b}(\nu_*M, \nu_*M)$ is independent of the choice of a basis for $\mathcal{O}_X^{\otimes R}$, and its rank is independent of $M$. Hence, $R^0 p_*\mathcal{E}xt^1_{C_b}(\mathcal{U}|_{Z}, \mathcal{U}|_{Z})$ is the pull-back of a vector bundles $A$ on $F_b$.

The transitive action of $\Pic^{d'}(C_b^{\nu})$ by translation induces identification of the fibers of $A$. Indeed, if $\mathcal{I}\coloneqq \nu_*M$ and $\mathcal{I}'\coloneqq \nu_*(M \otimes L)$ with $L \in \Pic^{d'}(C_b^{\nu})$, then $\mathcal{I}$ and $\mathcal{I}'$ are locally isomorphic, and so there exists an invertible sheaf $L'$ on $X$ such that $\mathcal{I}'=\mathcal{I}\otimes L'$ with $\nu^*L'=L$; see for instance \cite[Lemma 2.2.2 or 3.1.2/3.1.3]{Cook93}.\TBC{
 Notice that the invertible sheaf  $L'$ on $X$ is uniquely determined up to elements in the affine part of $\Pic^0(C_{b})$.}
Hence, we have
\[
\ext^1_{C_b}(\nu_*(M \otimes L), \nu_*(M \otimes L))=\ext^1_{C_b}(\nu_*M \otimes L', \nu_*M\otimes L')= \ext^1_{C_b}(\nu_*M, \nu_*M).
\]
This gives a global trivialization of $A$.  
\end{proof}

\begin{lem}\label{lem:triviality} 
Let $N_{F_b/\unJBGamma}$ be the normal bundle of $F_{b}$ in $\Pic^{d'}(\mathcal{C}^{\nu}) \subseteq \unJvers$ as constructed in \cref{rmk:analvshomeo}. Denote by $\Tot(N)$ the total space of a vector bundle $N$. Then 
\begin{equation*}
    \Tot(N_Z)\sslash \Sl_R \simeq \Tot(N_{F_b/\Pic^{d'}(\mathcal{C}^{\nu})}) \times X(\Gamma(C_b)^{\pm}, 0).
\end{equation*}
\end{lem}
\begin{proof}
By \cref{lem:splitnormalbundle} we have
\[\Tot(N_Z)\sslash \Sl_R \simeq (\Tot(N^{\mathbb{T}^{r}}_{Z}) \times_{Z} \Tot(R^0 p_*\mathcal{E}xt^1_{C_b}(\mathcal{U}|_{Z}, \mathcal{U}|_{Z}))\sslash \mathbb{T}^{r-1})\sslash(\Sl_{r}/\mathbb{T}^{r-1}),\]
and \cref{lem:trivialvectorbundle} gives \[
\Tot(R^0 p_*\mathcal{E}xt^1_{C_b}(\mathcal{U}|_{Z}, \mathcal{U}|_{Z})) \simeq Z \times \CC^{s}.
\]
The algebraic torus $\mathbb{T}^{r-1}=\mathbb{T}^{r}/\mathbb{T}$ acts on $\CC^{s}$ as prescribed in \eqref{eq: groupactionII}, which coincides with the torus action defining $X(\Gamma(C_b)^{\pm}, 0)$. Therefore, we obtain
\begin{align*}
  \Tot(R^0 p_*\mathcal{E}xt^1_{C_b}(\mathcal{U}|_{Z}, \mathcal{U}|_{Z}))\sslash \Sl_{r} & \simeq ((Z \times \CC^{s}) \sslash \mathbb{T}^{r-1})\sslash(\Sl_{r}/\mathbb{T}^{r-1})\\
  & \simeq Z\sslash(\Sl_{r}/\mathbb{T}^{r-1}) \times X(\Gamma(C_b)^{\pm}, 0) \simeq F_{b}\times X(\Gamma(C_b)^{\pm}, 0).
\end{align*}
Finally, the vector bundle $N^{\mathbb{T}^{r}}_{Z}$ descends to $N_{F_b/\Pic^{d'}(\mathcal{C}^{\nu})}$, because $\Pic^{d'}(\mathcal{C}^{\nu})$ parametrises the locally trivial deformations of the pair $(C_{b}, \mathcal{I})$.
\end{proof}

\subsection{Stratified Ehresmann's theorem}\label{sec:Ehresmann} We use a stratified version of Ehresmann's theorem to prove that $\mathfrak{n}^{-1}(0)$ is a model of a tubular neighbourhood of $F_{b}$ in $\unJvers$.

\begin{lem}\label{lem:localtriviality}
The degeneration     $\mathfrak{n}\colon \mathcal{Q}\sslash \Sl_R \to \CC$ 
is locally analytically trivial along $F_{b}$, i.e.\ for any $x \in F_{b} \subset \mathfrak{n}^{-1}(0) \subset \mathcal{Q} \sslash \Sl_R$ there exists a trivialising neighbourhood $U \subset \mathcal{Q} \sslash \Sl_R$ of $x$,
and a local analytic isomorphism  \begin{equation}\label{eq:psi}
    \psi \colon U \to ( \CC^{c'} \times X(\Gamma(C_b)^{\pm}, 0)) \times \CC,
\end{equation}
with $c'\coloneqq 4g'-3-b_1(\Gamma)-s$, such that
$\mathfrak{n} \colon U \to \CC$ corresponds to the third projection $(X(\Gamma(C_b)^{\pm}, 0) \times \CC^{c'}) \times \CC \to \CC$.
\end{lem}

\begin{proof}
For any $z \in Z$, the degeneration $\mathfrak{n} \colon Z \times \CC \subseteq \mathcal{Q} \to \CC$ admits a section $\{z\} \times \CC$. By \eqref{eq:Normal}, an analytic neighbourhood of $\{z\} \times \CC$ is locally isomorphic to
\[(T_{z}O_z \times T_{z}  \defo^{lt}(C_{b},\mathcal{I}) \times\bigoplus_{e \in \Sigma} T_{z}\defo(C_{b, e},\mathcal{I}_e)) \times \CC,\] where $O_z$ is the $\Sl_R$-orbit of $z$. By Luna's slice theorem, an analytic neighbourhood of $\{[z]\} \times \CC$ in $\mathcal{Q}\sslash \Sl_R$ is locally isomorphic to 
\begin{align*}
    (T_{z}  \defo^{lt}(C_{b},\mathcal{I}) \times  & \bigoplus_{e \in \Sigma}  T_{z} \defo(C_{b, e},\mathcal{I}_e))\sslash \mathbb{T}^{r-1} \times \CC \\
     & \simeq (T_{(C_b, \mathcal{I})}\Pic^{d'}(\mathcal{C}^{\nu}) \times X(\Gamma(C_b)^{\pm}, 0)) \times \CC \simeq ( \CC^{c'} \times X(\Gamma(C_b)^{\pm}, 0)) \times \CC.
\end{align*}

\end{proof}

\begin{proof}[Proof of \cref{prop:localtrivial}]
Let's first show that a relatively compact euclidean neighbourhood of $F_{b}$ in $
\mathfrak{n}^{-1}(0)$ is stratified diffeomorphic to a euclidean neighbourhood of $F_{b}$ in $
\mathfrak{n}^{-1}(0)$ in $\mathfrak{n}^{-1}(t)=\unJvers$, with $t \neq 0$. 

By \cref{lem:localtriviality} the degeneration $\mathfrak{n}\colon \mathcal{Q}\sslash \Sl_R \to \CC$ is locally trivial. This means that we can
cover $F_{b} \subset \mathfrak{n}^{-1}(0)$ by trivialising open sets $U_k$, with $k$ in the index set $K$, as in \cref{lem:localtriviality}.
Take now the vector field $v_k \coloneqq d\psi_k^{-1}(\partial_t)$ on $U_k$, where $\partial_t$ is a non-vanishing vector field on $\CC$ and $\psi_{k}$ is the stratified diffeomorphism \eqref{eq:psi}. Consider a partition of unity $\{\chi_k\}_{k \in K}$ subordinate to the open cover $\{U_k\}_{k \in K}$. Then the flow of the vector field $v \coloneqq \sum_{k \in K} \chi_k \cdot v_k$ defines the required stratified diffeomorphism.

We conclude now by noticing that the neighbourhood of $F_{b}$ in $
\mathfrak{n}^{-1}(0)$ can be chosen homeomorphic to the product 
\[
\Tot(N_{F_b/\Pic^{d'}(\mathcal{C}^{\nu})})  \times X(\Gamma(C_b)^{\pm}, 0) \underset{\text{homeo}}{\simeq}  F_{b} \times \CC^{3g'-3-s} \times X(\Gamma(C_b)^{\pm}, 0)\] by \cref{lem:triviality}.
\end{proof}

\section{Full support theorem for universal compactified Jacobians}\label{sec:fullsupport}

\begin{lem}[Weak abelian fibration]\label{lem:weakabelian}
Let $f \colon \mathcal{C} \to B$ be a flat projective versal family of semistable curves of arithmetic genus $g'$, and $\pi^P \colon P \coloneqq \Pic^0(\mathcal{C}/B) \to B$ be the relative degree 0 Picard scheme parametrising line bundles on the fibres of $f$ whose restrictions to each irreducible components are of degree $0$. Let $\pi\colon \unJvers \to B$ be the relative compactified Jacobian with respect to the canonical polarization. Then the triple $(\unJvers, P,B)$ is a weak abelian fibration.
\end{lem}

\begin{proof}
We need to check condition (1), (2) and (3) in \cref{defn:weakabelian}. The relative Picard scheme $\pi^P \colon P \to B$ has fibres of pure dimension $g'$, and $\dim \unJvers = g' + \dim B$; see for instance \cref{thm:localmodelJM}. Hence  Condition (1) holds. The proof of \cite[Lem. 3.4.5]{deCataldoRapagnettaSacca19} gives the affineness of the stabilisers, i.e.\ Condition (2). The same argument of \cite[\S 4.12]{Ngo2010} gives Condition (3); see also \cite[Thm 3.3.1]{deCataldo2017} where one can drop the assumption that $\mathcal{C}$ is a family of curves embedded in a surface. 
\end{proof}

\begin{prop}[Whitney stratification of $\unJvers$] \label{prop:Whitney} In the hypothesis of \cref{lem:weakabelian}, let $\mathcal{I}$ be a rank 1 torsion free sheaf on $\mathcal{C}_b \coloneqq f^{-1}(b)$ which fails to be locally free at the nodes $\Sigma$ of $\mathcal{C}_b$. Set $\Gamma_{\mathcal{C}_b}(\Sigma)$ to be the dual graph of any curve obtained from $\mathcal{C}_b$ by smoothing the nodes not in $\Sigma$. Then there exists a complex Whitney stratification
\[
\unJvers = \bigsqcup_{\Gamma} \unJBGamma,
\]
where $\unJBGamma$ is the locus of rank 1 torsion free sheaves in $\unJvers$ whose $\Gamma_{\mathcal{C}_b}(\Sigma)$ is isomorphic to $\Gamma$, and we have
\begin{equation}\label{eq:dimJcomparison}
    \codim_{\unJvers} \unJBGamma \geq 2 \codim_{\pi^{-1}(t)} (\unJBGamma \cap \pi^{-1}(t)).
\end{equation}
\end{prop}

\begin{proof}
The first statement is a reformulation of \cite[Thm A]{CMKV2015}.  
For the inequality \eqref{eq:dimJcomparison}, observe that if $\Gamma$ has $r$ vertices and $s$ edges, then \cite[Thm A]{CMKV2015} gives
\begin{align*}
    & \codim_{\unJvers}  \unJBGamma  = \dim X(\Gamma^{\pm}, 0) = 2s-r+1,\\
    & \codim_{\pi^{-1}(t)} (\unJBGamma \cap \pi^{-1}(t))= s-r+1.
\end{align*}
\end{proof}

\begin{lem}\label{lem:t>2dJac}
In the hypothesis of \cref{lem:weakabelian}, we have
\[\tau_{>2R}(R\pi_*\IC(\unJvers, \QQ))=0, \qquad R \coloneqq \dim \unJvers - \dim B=g'.\]
\end{lem}
\begin{proof}
It follows from \cref{prop:relative dimension bound general} and \cref{prop:Whitney}.
\end{proof}

\begin{lem}\label{lem:codim=delta}
In the hypothesis of \cref{lem:weakabelian}, any support $Z$ of $R\pi_*\IC(\unJvers, \QQ)$ satisfies
\[\codim Z = \delta_Z\]
\end{lem}
\begin{proof}
The inequality $\codim Z \geq \delta_Z$ appears for instance in \cite[Fact 2.4]{MSV2021}.\TBC{Does it follow also from \cite[Theorem 7.2.2]{Ngo2010}?} The reverse inequality $\codim Z \leq \delta_Z$ follows from \cref{lem:weakabelian}, \cref{lem:t>2dJac}, and \cite[Theorem 1.8]{MaulikShen2020II} or \cref{thm:Ngofreeness}. 
\end{proof}

Let $\pi \colon \unJvers \to B$ be the relative compactified Jacobian with respect to the canonical polarization of a versal deformation of semistable curves of arithmetic genus $g'$. We define the invariant
\begin{equation}\label{eq:Psi}
    \Psi_{g'} \coloneqq \dim \unJvers - 2 \dim B = 4{g'}-3-2(3g'-3)=3-2g'.
\end{equation}

\begin{thm}[Full support for Jacobians of versal families]\label{prop:fullsupportvers}
Let $f \colon \mathcal{C} \to B$ be a flat projective versal family of semistable curves and arithmetic genus $g'$, and $U \subset B$ the open subscheme over which the morphism $f$ is smooth. Let $\pi: \unJvers \to B$ be the relative compactified Jacobian with respect to the canonical polarization. Then the complex $R\pi_*\IC(\unJvers, \QQ)$ has full support on $B$, i.e.\
\begin{equation}\label{eq:fullsupportJac}
    R\pi_*\IC(\unJvers, \QQ) \simeq \bigoplus^{2g}_{l=0}\IC(B, \Lambda^l (R^1 \pi_*\QQ)|_U)[-l] = \bigoplus^{2g}_{l=0}\IC(B, \Lambda^l (R^1 f_*\QQ)|_U)[-l].
\end{equation}
In particular, for any $d, d' \in \ZZ$ there is an isomorphism of complexes of mixed Hodge modules
\begin{equation}\label{eq:indepdendenceJac}
    R\pi_* \IC(\unJvers, \QQ)\simeq R\pi_* \IC(\unJversp, \QQ).
\end{equation}
\end{thm}
\begin{proof}
The argument of \cite[\S 3]{MaulikShen2020II} can be adapted to work in the present setting.  
The base $B$
admits a stratification by type
\[B = \bigsqcup B^{\circ}_{ \underline{g}'},\] 
which extends \eqref{eq:typestratification}. For any $r\geq 1$, and
for any $r$-tuples of positive integers $\underline{g}'=(g'_i)$, the closed points $b \in B$ in $B^{\circ}_{\underline{g}'}$ corresponds to the fibers of $f$ of the form $\mathcal{C}_b=\bigcup_i \mathcal{C}_i$, where $\mathcal{C}_i$ are irreducible curves of genus $g'_i=g(\mathcal{C}_i)$. 

Assume that the irreducible subvariety $Z \subseteq B$ is a support of $R\pi_*\IC(\unJvers, \QQ)$, and that the generic point of $Z$ is contained in $B^{\circ}_{\underline{g}'}$.
As in \cite[Prop. 4.4]{MaulikShen2020II}\TBC{Do you agree that the same computation holds?}, the inequality $\codim Z \geq \delta_Z$ can be improved to 
\begin{equation}\label{eq:conditiona'}
    \Psi_{g'}+\codim Z \geq \sum^r_{i=1}\Psi_{g'_i}+\delta_Z.
\end{equation}
Combining \eqref{eq:conditiona'}, \eqref{eq:Psi} and \cref{lem:codim=delta}, we obtain
\[3(r-1)+2(g-\sum^r_{i=1}g_i)=0.\]
The only possibility is that $r=1$. Hence, $Z=B$, which implies \eqref{eq:fullsupportJac} and \eqref{eq:indepdendenceJac}.
\end{proof}

\begin{rmk}\label{rmk:comparisonmeromhol} The proof of \cref{prop:fullsupportvers} does not work for the Hitchin fibration. In this case, the invariant $\Psi = \dim M(n,d)-2 \dim A_n$ vanishes, and the inequality \eqref{eq:conditiona'} is inconclusive.
\end{rmk}

\section{Proof of the main theorems}\label{sec:summary}

\begin{proof}[Proof of \cref{thm:IHde} (Hodge-to-singular correspondence)]
In view of the decomposition \eqref{eq:decthmNgo}, it suffices to show that for any partition $\underline{n}$ of $n$ the isomorphism \eqref{eq:HtS} holds on an open euclidean neighbourhood of a general $a \in S^{\times}_{\underline{n}}$ containing a Zariski-dense open set of $S^{\times}_{\underline{n}}$. By \cref{rmk:Aa}, it is actually enough to show the isomorphism on the subsets $A_{a}$ constructed in the proof of \cref{prop:unramified}.

By \cref{prop:unramified}, proper base change gives
\[R\chi(n,d)_* (i_{M}^*\IC(\unJB, \QQ)) \simeq (f^{\times}|_{A_{a}})^*R\pi_*\IC(\unJB, \QQ).\]
 For any \'{e}tale chart $B \to \mathcal{B}$, \cref{prop:fullsupportvers} says that $R\pi_*\IC(\unJverse, \QQ)$ and $R\pi_*\IC(\unJvers, \QQ)$ have full support, which implies that
\[(f^{\times}|_{A_{a}})^*R\pi_*\IC(\unJBe, \QQ)\simeq (f^{\times}|_{A_{a}})^*R\pi_*\IC(\unJB, \QQ).\]
Since $\chi(n,e)^{-1}(A_a)$ is regularly embedded in $\unJBe$, we have
\[i_{M}^*\IC(\unJBe, \QQ) \simeq \QQ_{\chi(n,e)^{-1}(A_a)} \]
Together with \cref{prop:perverse} and \cref{cor:trivialityalongfibers}, we obtain \cref{thm:IHde}. 
\end{proof}

\begin{proof}[Proof of \cref{cor:degree1}]
By \cref{prop:perverse}, the descent of $\mathcal{L}'_{\underline{n}}$ on $S^{\times}_{\underline{n}}$ in \cref{cor:trivialityalongfibers} is a sublocal system of the sheaf $\mathscr{L}_{\underline{n}}(e)$ in the Ng\^{o} strings of \cref{thm:Ngostring}. By \cite[Thm 6.11]{deCataldoHeinlothMigliorini19},  \cref{FactToricQuiver}.\eqref{item:8} and \cref{prop:perverse},
they have the same rank
\[\rank(\mathscr{L}_{\underline{n}}(e))=\dim \widetilde{H}^{2b_1-1}(|\mathscr{C}_Q|, \QQ)= \dim H^{2b_1}(Y(Q, \theta), \QQ)= \rank(\mathcal{L}'_{\underline{n}}),\]
so they must coincide.
\end{proof}

\begin{proof}[Proof of \cref{thm:Ngostring0}]
\cref{thm:IHde} for $d=0$ yields
\begin{equation} \label{eq:mainthm0}
R\chi(n,e)_*\QQ_{M(n,e)}|_{A^{\mathrm{red}}_n}\simeq \bigoplus_{\underline{n}\vdash \, n} R\chi(n,0)_*\IC(M_{\underline{n}}(0),\chi(n,0)^*\mathcal{L}_{\underline{n}})[2(\dim S_n - c(g,n))]|_{A^{\mathrm{red}}_n}.
\end{equation}
On the open set of $S_{\underline{n}}$ over which the Hitchin map $\chi(n,0)\colon M_{\underline{n}}(0)\to S_{\underline{n}}$ is smooth, there are isomorphisms
\[R\chi(n,0)_*\IC(M_{\underline{n}}(0),\chi(n,0)^*\mathcal{L}_{\underline{n}})\simeq \bigoplus^{2 \dim S_{\underline{n}}}_{l=0} \Lambda^l_{\underline{n}}\otimes \mathcal{L}_{\underline{n}}[-l]\simeq \mathscr{S}_{\underline{n}}.\]
Therefore, by the decomposition theorem \cite{BeilinsonBernsteinDeligne1981, Saito89}, the string $\mathscr{S}_{\underline{n}}$ is a direct summand of \[R\chi(n,0)_*\IC(M_{\underline{n}}(0),\chi(n,0)^*\mathcal{L}_{\underline{n}})\] on the whole $S_{\underline{n}}$. Now \cref{cor:degree1} gives
\[
\bigoplus_{\underline{n}\vdash \, n} \mathscr{S}_{\underline{n}}|_{A^{\mathrm{red}}_n}  \langle \codim S_{\underline{n}} \rangle   \simeq R\chi(n,e)_*\QQ_{\MDol(n,e)}|_{A^{\mathrm{red}}_n} \simeq \text{ RHS of }\eqref{eq:mainthm0} \supseteq \bigoplus_{\underline{n}\vdash \, n} \mathscr{S}_{\underline{n}}|_{A^{\mathrm{red}}_n} \langle \codim S_{\underline{n}} \rangle.
\]
This implies that the complexes $R\chi(n,0)_*\IC(M_{\underline{n}}(0),\chi(n,0)^*\mathcal{L}_{\underline{n}})$ have full support on $S_{\underline{n}}$. For $\underline{n}=\{n\}$, this gives that $R\chi(n,0)_*\IC(\MDol(n,0),\QQ)$ has full support.
\end{proof}

\begin{proof}[Proof of \cref{thm:Fullsupportranktwo}]
By \cref{cor:degree1} and \cref{thm:Ngostring0}, it suffices to check that no summand of $R\chi(2,1)_*\QQ_{\MDol(2,1)}$ and  $R\chi(2,0)_*\IC({\MDol(2,0)},\QQ)$ is supported over $A_2 \setminus A^{\mathrm{red}}_2=\{0\}$. This is indeed the case by \cref{cor:vanishingintersectionform} (\cref{cor:van}).
\end{proof}

\begin{rmk}\label{rmk:repIH}
Let $\underline{n}=(n_1, \ldots, n_{r})=1^{\alpha_1}\cdot \ldots \cdot n^{\alpha_{n}}$ be a partition of $n$, where $\alpha_i$ is the number of elements in $n$ equal to $i$. Denote by $\mathfrak{S}_{\underline{n}} \coloneqq \prod^r_{i=1} \mathfrak{S}_{\alpha_{i}}$ the subgroup of the symmetric group $\mathfrak{S}_{r}$ stabilizing $\underline{n}$.
Then the horizontal arrows of the following commutative square
\begin{equation*}
\begin{tikzpicture}[baseline= (a).base]
\node[scale=1] (a) at (0,0){
\begin{tikzcd}
\prod^{r}_{i=1} \MDol(n_i,d_i) \ar[r]\ar[]{d}[swap]{\prod^r_{i=1}\chi(n_i,d_i)} & M_{\underline{n}}(d) \ar[]{d}{\chi(n,d)} & & (\mathcal{E}_i, \phi_i)^r_{i=1} \ar[r, mapsto]\ar[d, mapsto]& (\bigoplus^r_{i=1}\mathcal{E}_i, \bigoplus^r_{i=1}\phi_i)\ar[d, mapsto]\\
    \prod^r_{i=1}A_{n_i}\ar[r, "\mathrm{mult}_{\underline{n}}"]& S_{\underline{n}}, & & (\mathrm{char}(\phi_i))^r_{i=1} \ar[r, mapsto]& \prod^{r}_{i=1}\mathrm{char}(\phi_i)
\end{tikzcd}
};
\end{tikzpicture}
\end{equation*}
are quotient maps by the action of $\mathfrak{S}_{\underline{n}}$. By \cite[Cor. 6.20]{deCataldoHeinlothMigliorini19}, the monodromy of the local systems $\mathcal{L}_{\underline{n}}$ is given
by the restriction to the subgroup $\mathfrak{S}_{\underline{n}}$ of the representation of $\mathfrak{S}_{r}$ induced by a primitive character of a maximal cyclic subgroup. In particular,  $\mathrm{mult}^{*}_{\underline{n}}\mathcal{L}_{\underline{n}}$ is a trivial local system. Hence, $\mathfrak{S}_{\underline{n}}$ acts on the sheaf $\bigboxtimes^{r}_{i=1}R \chi(n_i,d_i)_* \IC(M(n_i,d_i), \QQ) \otimes \QQ^{\rank \mathcal{L}_{\underline{n}}}$ via the $\mathfrak{S}_{\underline{n}}$-representation  $\mathcal{L}_{\underline{n}}$.
This means that the summand of the RHS of \eqref{eq:HtS} are the $\mathfrak{S}_{\underline{n}}$-invariant part of \begin{equation}\label{eq:summandexplicit}
    \bigboxtimes^{r}_{i=1}R \chi(n_i,d_i)_* \IC(M(n_i,d_i), \QQ) \otimes \QQ^{\rank \mathcal{L}_{\underline{n}}} \langle \codim S_{\underline{n}} \rangle.
\end{equation} 
In particular, the summands of the RHS of \eqref{eq:decompglobal} are $\mathfrak{S}_{\underline{n}}$-subrepresentations of 
\[\bigotimes^{r}_{i=1} \IH^{k - 2\codim S_n}(M(n_i, d_i)^{\mathrm{red}}, \QQ).\]
\end{rmk}

\begin{proof}[Proof of \cref{cor:independence}] The summation indices and the summands in \eqref{eq:HtS} 
depends only on $\gcd(n,d)$.
\end{proof}

\subsection{Dependence of $\IH^*(\MDol(n,d), \QQ)$ on degree}
\label{rmk:dependence}
The description of the summands of the decomposition theorem for the Hitchin fibration in the intermediate case, $\gcd(n,d)\neq 0$ or $\neq 1$, is more subtle. A close formula appears now in \cite[Thm 1.1, Thm 1.6, Thm 1.8]{MMP2022}, but \eqref{eq:HtS} offers already a recursive approach. Combining \cref{thm:IHde} and \cref{thm:Ngostring0}, one obtain the isomorphism in $D^bMHM_{\mathrm{alg}}(A^{\mathrm{red}}_n)$
\begin{equation}\label{eq:dec} \small
    R\chi(n,d)_*\IC(\MDol(n,d), \QQ) \oplus \bigoplus_{\substack{\underline{n} \in \mathcal{P}_{d}, \, \underline{n} \neq \{n\}}} R\chi(n,d)_*\IC(M_{\underline{n}}(d),\chi(n,d)^*\mathcal{L}_{\underline{n}})\langle \codim S_{\underline{n}}\rangle \simeq \bigoplus_{\underline{n}=\{n_i\}\vdash \, n}  \mathscr{S}_{\underline{n}}.
\end{equation}
Recursively from this identity, one can determine the Ng\^{o} strings for $R\chi(n,d)_*\IC(\MDol(n,d), \QQ)$, and the subvarieties $S_{\underline{n}}$ which are actual support of $R\chi(n,d)_*\IC(\MDol(n,d), \QQ)$. We refer to \cite{MMP2022} for the explicit computation of the ranks of the Ng\^{o} strings. Here we limit to illustrate the first non-trivial example $n=4$. More care is needed to determine explicitly the monodromy of the associated local systems.

\begin{exa}[$n=4$] We determine the rank of the local systems $\mathscr{L}_{\underline{n}}(d)$ in \cref{thm:Ngostring} for $n=4$. To this end, observe that the Hodge-to-singular correspondence gives
\begin{equation}\label{eq:explicit}
    R\chi(4,2)_*\IC(\MDol(4,2), \QQ) \oplus R\chi(4,2)_*\IC(M_{\{2,2\}}(2),\chi(4,2)^*\mathcal{L}_{\{2,2\}})\langle 8g-9\rangle \simeq \bigoplus_{\underline{n}=\{n_i\}\vdash \, 4}  \mathscr{S}_{\underline{n}}.
\end{equation}
Consider the commutative square
\begin{equation*}
\begin{tikzpicture}[baseline= (a).base]
\node[scale=1] (a) at (0,0){
\begin{tikzcd}
\MDol(2,1)^2 \ar[r]\ar[]{d}[swap]{\chi(2,1)^2} & M_{\{2,2\}}(2) \ar[]{d}{\chi(4,2)}\\
    A_{2}^2\ar[r, "\mathrm{mult}_{\{2,2\}}"]& S_{\{2,2\}}.
\end{tikzcd}
};
\end{tikzpicture}
\end{equation*}
The map $\mathrm{mult}_{\{2,2\}}$ factors through the normalization $\nu_{\{2,2\}} \colon \mathrm{Sym}^2 A_2 \to S_{\{2,2\}}$ of $S_{\{2,2\}}$ as follows:
\begin{equation*}
   \mathrm{mult}_{\{2,2\}} \colon  A^2_{2} \xrightarrow{\eta_{\{2,2\}}}  \mathrm{Sym}^{2}A_{2} \xrightarrow{\nu_{\{2,2\}}} S_{\{2,2\}}.
\end{equation*}
Set $\underline{n}= \{2,2\}, \{2,1,1\}$ or $\{1,1,1,1\}$. The degree of $\nu_{\{2,2\}}$ over $S_{\underline{n}}$, denoted by $d_{\underline{n}}$, is 
\[
d_{\underline{n}} = \begin{cases}
1 \quad \text{ for }\underline{n}= \{2,2\}, \{2,1,1\} \text{, and} \\
3 \quad \text{ for } \underline{n}=\{1,1,1,1\}.
\end{cases}
\]
In fact, the invariant $d_{\underline{n}}$ counts in how many ways one can sum the elements of the partition $\underline{n}$ to get $\{2,2\}$. The map $\eta_{\{2,2\}}$ instead is \'{e}tale at the general points of $S_{\underline{n}}$. By the behaviour of intersection cohomology under normalization and \'{e}tale maps, we obtain that $R\chi(2,4)_*\IC(M_{\{2,2\}}(2),\chi(4,2)^*\mathcal{L}_{\{2,2\}})$ is locally isomorphic to 
\begin{equation}\label{eq:normal}
    R\chi(2,1)^2_*\IC(\MDol(2,1)^2, \QQ) \otimes \QQ^{\rank \mathcal{L}_{\{2,2\}}} \otimes \QQ^{d_{\underline{n}}}
\end{equation} at the general points of $S_{\underline{n}}$.
We omit the tensorisation by $\otimes  \QQ^{\rank \mathcal{L}_{\{2,2\}}}$ since $\rank \mathcal{L}_{\{2,2\}}=1$. Indeed, \begin{equation}
    \mathscr{L}_{\underline{n}}(d) \leq \rank \mathcal{L}_{\underline{n}}=(|\underline{n}|-1)!
\end{equation} by \cref{thm:IHde} and \cite[Cor. 6.20]{deCataldoHeinlothMigliorini19}.
By the K\"{u}nneth formula and \cref{thm:Fullsupportranktwo}, we have
\[
R\chi(2,1)^2_*\IC(\MDol(2,1)^2, \QQ) \simeq (\mathscr{S}_{\{2\}} \oplus \mathscr{S}_{\{1,1\}}) \boxtimes (\mathscr{S}_{\{2\}} \oplus \mathscr{S}_{\{1,1\}}),
\]
so all the Ng\^{o} strings of $R\chi(2,1)^2_*\IC(\MDol(2,1)^2, \QQ)$ start with a local system of rank one. By \eqref{eq:explicit} and \eqref{eq:normal}, this implies that $\rank \mathscr{L}_{\underline{n}}(2) = \rank \mathcal{L}_{\underline{n}}-d_{\underline{n}}$. In particular, $S_{\{2,2\}}$ is not a support of $R\chi(2,4)_*\IC(\MDol(2,4), \QQ)$. We summarise the results in Table \ref{table}.

\begin{table}[h]
\centering
\renewcommand{\arraystretch}{1}
   \begin{tabular}{c|ccccc}
  $\operatorname{gcd}(n,d)$  \quad & \{4\} & \{3,1\} & \{2,2\} & \{2,1,1\} & \{1,1,1,1\}\\ \hline
  0 & 1 & 0 & 0 & 0 & 0 \\
  1 & 1 & 1 & 1 & 2 & 6 \\
  2 & 1 & 1 & 0 & 1 & 3
    \end{tabular}
    \vspace{0.2 cm}
    \caption{The entries of the table are the rank of the local systems $\mathscr{L}_{\underline{n}}(d)$ for $n=4$.}   \label{table}
    \end{table}
\end{exa}

\cref{thm:Ngostring} says that the complex $R \chi(n,d)_* \IC(\MDol(n,d), \QQ)$ is a direct sum of Ng\^{o} strings
\[
\mathscr{S}(\mathscr{L}_{\underline{n}}(d)) \coloneqq \bigoplus^{2 \dim S_{ \underline{n}}}_{l=0} \IC(S_{\underline{n}}, \Lambda^l_{\underline{n}} 
\otimes \mathscr{L}_{ \underline{n}}(d))[-l]\langle \codim S_{\underline{n}}\rangle. 
\]
By \cref{rmk:repIH}, the summands $R\chi(n,d)_*\IC(M_{\underline{n}}(d),\chi(n,d)^*\mathcal{L}_{\underline{n}})$ of the LHS of \eqref{eq:dec} are the $\mathfrak{S}_{\underline{n}}$-invariant part of \[
    \bigboxtimes^{r}_{i=1}R \chi(n_i,d_i)_* \IC(M(n_i,d_i), \QQ) \otimes \QQ^{\rank \mathcal{L}_{\underline{n}}},\] where $n_i < n$ and $d_i < d$.
   Therefore, using \eqref{eq:dec}, one could rewrite the local system $\mathscr{L}_{\underline{n}}(d)$ as a combinations of direct sums, linear quotients and $\mathfrak{S}_{\underline{n}'}$-quotients of $\mathscr{L}_{\underline{n}'}(d')$, $\mathcal{L}_{\underline{n}}$ and $\mathcal{L}_{\underline{n}'}$, with $n' < n$ and $d' < d$. Recursively, we could remove the dependence of $\mathscr{L}_{\underline{n}}(d)$ on $\mathscr{L}_{\underline{n}'}(d')$, and write $\mathscr{L}_{\underline{n}}(d)$ as depending only on $\mathcal{L}_{\underline{k}}$ with $k \leq n$. For any partition $\underline{n}=(n_1, \ldots, n_r)$, \cref{rmk:repIH} says that the Ng\^{o} string associated to $\mathcal{L}_{\underline{n}}$, namely
   \[
\mathscr{S}_{\underline{n}} \coloneqq \bigoplus^{2 \dim S_{\underline{n}}}_{l=0} \IC(S_{\underline{n}}, \Lambda^l_{\underline{n}}
\otimes \mathcal{L}_{\underline{n}})[-l]\langle \codim S_{\underline{n}} \rangle,\]
is the $\mathfrak{S}_{\underline{n}}$-invariant part of $
    (\bigboxtimes^{r}_{i=1}  \mathscr{S}_{n_i})^{\rank \mathcal{L}_{\underline{n}}}$.
    This means that $R \chi(n,d)_* \IC(\MDol(n,d), \QQ)$ can be expressed in terms of the Ng\^{o} strings $\mathscr{S}_{k} \simeq R \chi(k,0)_* \IC(\MDol(k,0), \QQ)$ with $k \leq n$. In other words, one obtain a recursive (though highly nontrivial) formula for $IH^*(\MDol^{\text{red}}(n,d), \QQ)$ in terms of $IH^*(\MDol^{\text{red}}(k,0), \QQ) = H^*(\mathscr{S}_{k})$ with $k \leq n$. 

\subsection{Some applications of the Hodge-to-singular correspondence}
\begin{proof}[Proof of \cref{cor:stab}] Taking global cohomology in \eqref{eq:decthmNgo}, we have
\[\IH^{\leq k}(M(n,d))\simeq \IH^{\leq k}(M(n,0))\]
with $k = \min \{ 2(\dim \mathcal{A}_n - \dim \mathcal{A}^{\text{red}}_n), \dim M_{\underline{n}} \colon \underline{n}\neq \{n\}\}$. By \cite[Lem. 2.2]{BellamySchedler2019}, we have 
\begin{align}
    k & = 2(n^2(g-1)+1)-2\max \bigg\{ r+(g-1)\sum^{r}_{i=1} n^2_{i}\, \bigg| \, m_i, n_i, \sum^r_{i=1} m_i n_i = n \bigg\}\label{eq:codimension}\\
    & =4(g-1)(n-1)-2 = O(n). \nonumber
\end{align}
This gives the independence of the limits \eqref{eq:limstab} by the degree. The finiteness follows from the generation of the cohomology of $M(n,e)$ with $\gcd(e,n)=1$ by finitely many tautological classes in \cite[Thm 7]{Markman02}.
\end{proof}
Recall that the perverse filtration on $\IH^*(M(n,d))$ associated to the Hitchin map $\chi(n,d)$ is given by
\begin{equation}\label{eq:perversefiltration}
    P_j \IH^k(M(n,d), \QQ)=\ker\{\IH^k(M(n,d), \QQ) \to \IH^k(\chi(n,d)^{-1}(\Lambda^{k-j-1}), \QQ)\},
\end{equation}
where $\Lambda^s$ is a general $s$-dimensional affine hyperplane of $A_{n}$; see \cite[Thm 4.1.1]{deCataldoMigliorini2010}. By \eqref{eq:codimension}, $\Lambda^s$ does not intersect the supports of the RHS of \eqref{eq:HtS} which are properly contained in $A_n$, provided that $s<2(g-1)(n-1)-1$. As a result, we obtain the following proposition.
\begin{prop}\label{prop:perverseindependet}
The graded piece of the perverse filtration
\[\Gr^P_j \IH^k(M(n,d), \QQ) \coloneqq P_j \IH^k(M(n,d), \QQ)/P_{j-1} \IH^k(M(n,d), \QQ)\]
are independent of the degree $d$ for $j>k-2(g-1)(n-1)+1$.
\end{prop}

\begin{proof}[Proof of \cref{prop:restr}]
By \eqref{eq:perversefiltration} we have
\[\Gr^P_k \IH^k(M(n,d), \QQ)\simeq \mathrm{Im}\{\IH^k(M(n,d), \QQ) \to H^k(M_a, \QQ)\}.\]
Let $e$ be an integer coprime with $n$. Since \[ \Gr^P_k \IH^k(M(n,d), \QQ) \simeq \Gr^P_k H^k(M(n,e), \QQ)\] by \cref{prop:perverseindependet}, it suffice to determine the latter. 
Let $(\mathbb{E}, \Phi)$ be the universal family of Higgs bundles on $C \times M(n,e)$ normalised such that 
\[\mathrm{ch}_1(\mathbb{E})|_{p \times M(n,e)}=0 \in H^2(M(n,e), \QQ), \qquad \mathrm{ch}_1(\mathbb{E})|_{C \times q}=0 \in H^2(C, \QQ). \]
For any $\gamma \in H^j(C, \QQ)$, the tautological class $c(\gamma, k)$ is the integral
\[
\int_{\gamma} \mathrm{ch}^{k}(\mathbb{E}) \in H^{j+2k-2}(M(n,d), \QQ).
\]
By \cite[Thm 7]{Markman02}, the cohomology $H^k(M(n,e), \QQ)$ is generated by tautological classes as $\QQ$-algebra. The proof of \cite[Cor. 5.1.3]{deCataldoHauselMigliorini2012} works in arbitrary rank, and shows that the only tautological classes that do not vanish along $M_a$ are $c(\gamma, 1)$ with $\gamma \in H^1(C, \QQ)$ and the $\chi(n,e)$-relative ample class $\alpha \coloneqq c([c],2)$, where $[c]$ generates $H^2(C, \QQ)$.
We conclude that  
\begin{equation}\label{eq:ident}
    \Gr^P_{*} H^{*}(M(n,e), \QQ) = \langle c(\gamma, 1)|_{M_a}, \alpha|_{M_a} \rangle  = H^{\bullet}(C, \QQ) \otimes  {\QQ[\alpha|_{M_a}]}/{(\alpha|_{M_a}^{\dim M(n,d)+1})}.
\end{equation}
The result for $\check{M}(n,d)$ is an immediate corollary of \eqref{eq:ident} since \[\Gr^P_{*} H^{*}(M(n,d), \QQ) \simeq \Gr^P_{*} H^{*}(\check{M}(n,d), \QQ) \otimes H^{\bullet}(C, \QQ);\]
see for instance \cite[\S 4.4]{deCataldoHauselMigliorini2012}.
\end{proof}

\subsection{Mukai moduli spaces}\label{sec:Mukai} Let $S$ be a K3 surface. Given an effective Mukai vector\footnote{i.e.\ there exists a coherent sheaf $\mathcal{F}$ on $S$ such that $v = (rk(\mathcal{F}), c_1(\mathcal{F}), \chi(\mathcal{F})-rk(\mathcal{F}))$.} $v \in H^{*}_{\mathrm{alg}}(S, \ZZ)$, we denote by $\Mukai$ the moduli space of Gieseker semistable sheaves on $S$ with Mukai vector $v$ with respect to a polarization $H$; see \cite[\S 1]{Simpson1994I}. If $v$ is the Mukai vector of a pure dimensional one sheaf $\mathcal{F}$, then $v=(0, D, \chi(\mathcal{F}))$, where $D \in \Pic(S, \ZZ)$ is the class of the curve supporting $\mathcal{F}$.
Taking
(Fitting) supports defines a Lagrangian fibration
\[\Mukai \to |D|, \qquad \mathcal{F} \mapsto \mathrm{Supp}(\mathcal{F}),\]
whose fibers are compactified Jacobians. The Hodge-to-singular correspondence extends to this context with no change, so does \cref{cor:stab}. 

\begin{cor}\label{cor:Pic1}
Let $S$ be a K3 surface with $\Pic(S) \simeq \ZZ$ generated by the class of a smooth curve $C$ of genus $g \geq 2$. Then 
$\lim_{n \to \infty} \dim \IH^{p,q}(M(S, (0, nC, \chi), C))$
is finite and independent of $\chi$.
\end{cor}
\begin{proof}
The same proof of \cref{cor:stab} holds in this context too. The assumption on the Picard rank grants that the codimension of the locus of non-reduced curves in $|nC|$ grows linearly in $n$.
\end{proof}

\begin{defn}\cite[Def. 1.15]{PeregoRapagnetta18}
 We say that $(S,v,H)$ is a $(m.k)$-triple such that:
 \begin{enumerate}
     \item the polarization $H$ is primitive and $v$-generic in the sense of \cite[\S 2.1]{PeregoRapagnetta18};
     \item we have $v=mw$, where $w$ is primitive and $w^2=2k$;
     \item if $w=(0, w_1, w_2)$ and $\rho(S)>1$, then $w_2 \neq 0$.
 \end{enumerate}
\end{defn}

Note that the stabilization of the cohomology of $\Mukai$ (\cite[Conjecture 1.1]{CK2018}) is known in the literature if $v$ is primitive; see \cite[p. 4]{CK2018} and references therein. We can now drop the primitivity assumption.

\begin{cor} Let $v^2= m w^2$ and $w^2=2k$, with $v, w \in H^{*}_{\mathrm{alg}}(S, \ZZ)$, $m, k>0$. If $H$ is primitive and $v$-general, the intersection Betti and Hodge numbers of $\Mukai$ stabilize to the
stable Betti and Hodge numbers of the Hilbert scheme of $s$-points as $s$ 
tends to infinity, i.e.\ the limits
$\lim_{m \to \infty} \dim \IH^{p,q}(\Mukai)$ are finite and independent of $k$.
\end{cor}
\begin{proof}
Given two $(m,k)$-triples $(S_1, v_1, H_1)$ and $(S_2, v_2, H_2)$, $M(S_1, v_1, H_1)$ and $M(S_2, v_2, H_2)$ are deformation equivalent, and the deformation is locally analytically trivial by \cite[Thm 1.17]{PeregoRapagnetta18}. In particular, we have
$\dim \IH^{p,q}(M(S_1, v_1, H_1))= \dim\IH^{p,q}(M(S_2, v_2, H_2))$.

Now choose a K3 surface $S_0$ with $\Pic(S_0) \simeq \ZZ$ generated by the class of a smooth curve $C$ of genus $g=k+1$. Since $(S, v, H)$ and $(S_0, (0,mC, m), C)$ are $(m, g-1)$-triples,  we have \[\dim \IH^{p,q}(M(S, v, H))= \dim\IH^{p,q}(M(S_0, (0,mC, m), C)).\] By \cref{cor:Pic1}, we have 
\[\lim_{m \to \infty} \dim \IH^{p,q}M(S_0, (0,mC, m), C))=\lim_{m \to \infty} \dim \IH^{p,q}(M(S_0, (0,mC, 1), C)),\]
where $M(S_0, (0,mC, 1)$ is deformation equivalent to a Hilbert scheme of $s \coloneqq 2m^2(g-1)$ points on a K3 surface.
\end{proof}

\bibliographystyle{plain}
\bibliography{construction}
\end{document}